\documentclass[]{siamart1116}

\usepackage{lipsum}
\usepackage{amsopn}
\usepackage{amsfonts}
\usepackage{graphicx}
\usepackage{epstopdf}
\usepackage{algorithmic}
\ifpdf
  \DeclareGraphicsExtensions{.eps,.pdf,.png,.jpg}
\else
  \DeclareGraphicsExtensions{.eps}
\fi

\usepackage[utf8]{inputenc}
\usepackage{amsmath}
\usepackage{amssymb}
\usepackage{mathtools}
\usepackage{physics}
\usepackage{mathrsfs}
\usepackage{tikz}
\usepackage{algorithm}

\usepackage{microtype}
\usepackage{comment}
\usepackage{booktabs}
\usetikzlibrary{arrows}
\usepackage[a4paper, total={6in, 8in}]{geometry}

\numberwithin{theorem}{section}
\newtheorem{notation}[theorem]{Notation}
\newtheorem{remark}[theorem]{Remark}
\providecommand{\assumptionnumber}{}
\makeatletter
\newenvironment{asssumption}[2]
 {%
  \renewcommand{\assumptionnumber}{Assumption $\mathcal{#1}$#2}%
  \begin{assumption*}%
  \protected@edef\@currentlabel{$\mathcal{#1}$#2}%
 }
 {%
  \end{assumption*}
 }
\makeatother
\newtheorem*{assumption*}{\assumptionnumber}

\newcommand{\TheTitle}{Extending CSR decomposition to tropical inhomogeneous matrix products} 
\newcommand{\TheAuthors}{A. Kennedy-Cochran-Patrick and S. Sergeev}

\headers{\TheTitle}{\TheAuthors}

\title{{\TheTitle}\thanks{Submitted to the editors on September 9, 2020.
\funding{This work was supported by EPSRC Grant EP/P019676/1.}}}

\author{
  Arthur Kennedy-Cochran-Patrick\thanks{School of Mathematics, University of Birmingham, Birmingham, UK (\email{axc381@bham.ac.uk}).}
  \and
  Serge{\u{\i}} Sergeev\thanks{School of Mathematics, University of Birmingham, Birmingham, UK (\email{s.sergeev@bham.ac.uk}).}
}

\ifpdf
\hypersetup{
  pdftitle={\TheTitle},
  pdfauthor={\TheAuthors}
}
\fi

\def\digr{\mathcal{D}}
\def\trellis{\mathcal{T}}
\def\trellisgammak{\trellis(\Gamma(k))}
\def\nodes{N}
\def\edges{E}
\def\trellisnodes{\mathcal{N}}
\def\trellisedges{\mathcal{E}}
\def\Rmax{\mathbb{R}_{\max}}
\def\critnodes{\mathcal{N}_{c}}
\def\critedges{\mathcal{E}_c}
\def\crit{\mathbf{C}}
\def\cclass{\mathcal{C}}
\def\subcrit{\mathcal{G}}
\def\wiel{\operatorname{Wi}}
\def\sch{\operatorname{Sch}}
\def\Asup{A^{\sup}}
\def\Ainf{A^{\inf}}
\def\diag{\operatorname{diag}}
\def\init{\operatorname{init}}
\def\final{\operatorname{final}}
\def\full{\operatorname{full}}
\def\tropzero{\varepsilon}

\DeclareMathOperator{\lcm}{lcm}

\newcommand{\modd}[1]{(\operatorname{mod}{#1})}
\newcommand{\walkslen}[4]{\mathcal{W}^{#3}_{#4}(#1\to #2)}
\newcommand{\walks}[3]{\mathcal{W}_{#3}(#1\to #2)}
\newcommand{\walkslennode}[5]{\mathcal{W}^{#3}_{#4}(#1\xrightarrow{#5} #2)}
\newcommand{\walksnode}[4]{\mathcal{W}_{#3}(#1\xrightarrow{#4} #2)}

\begin{document}

\maketitle

\begin{abstract}
This article presents an attempt to extend the CSR decomposition, previously introduced for tropical matrix powers, to tropical inhomogeneous matrix products. The CSR terms for inhomogeneous matrix products are introduced, then a case is described where an inhomogenous product admits such CSR decomposition after some length and give a bound on this length. In the last part of the paper a number of counterexamples are presented to show that inhomogenous products do not admit CSR decomposition under more general conditions.    
\end{abstract}



\section{Introduction}
Tropical (max-plus) linear algebra is the linear algebra developed over the set $\Rmax=\mathbb{R}\cup\{ -\infty \}$ equipped with the additive operator $\oplus: a\oplus b = \max(a,b)$ and the multiplicative operator $\otimes: a \otimes b = a+b$. We will be working with the max-plus multiplication of matrices $A \otimes B$ defined by the operation
\begin{align*}
    (A \otimes B)_{i,j} = \bigoplus_{1 \leq k \leq n} a_{i,k} \otimes b_{k,j} = \max_{1\leq k \leq n}(a_{i,k}+b_{k,j})
\end{align*}
using two matrices $A=(a_{i,j})$ and $B=(b_{i,j})$ of appropriate sizes.

Consider the tropical dynamic system of equations given by
\begin{align*}
    x(0)& =x_{0} \\
    x(k)& =x(k-1) \otimes A_{k} \quad \text{ for }k\geq 1 \\ 
    x(k)& =x_{0} \otimes A_{1} \otimes \ldots \otimes A_{k} = x_{0} \otimes \Gamma(k).
\end{align*}
Here the matrices $A_{i}$ are taken in some unspecified order from a possibly infinite set of matrices $\mathcal{X}$. In practical terms, this represents a dynamical system where some accidental changes may occur over time. This has useful applications in modelling scheduling systems that are subject to change.

Much work has been done for the case where each matrix $A_{i}$ is the same for each step. Cohen et al.~\cite{CDQV-83,CDQV-85} were the first to observe that, under some mild conditions, the tropical powers $\{A^t\}_{t\geq 1}$ become periodic after a big enough time. A number of bounds on the transient of such periodicity were then obtained, in particular, by Hatmann and Arguelles~\cite{HA-99}, Akian et al.~\cite{AGW-05}, and Merlet et al.~\cite{MNS,MNSS}. In particular, Merlet et al.~\cite{MNS} offer an approach based on the CSR decompositions and CSR expansions of tropical matrix powers introduced by Sergeev and Schneider~\cite{Ser-09,SerSch}. Let us note that a preliminary version of such decompositions was introduced and studied before by Nachtigall~\cite{Nacht} and Moln\'{a}rov\'{a}~\cite{Mol-05}, and that similar decompositions appear in Akian et al.~\cite{AGW-05}.  

It is difficult to speak of ultimate periodicity in the case of inhomogeneous products. However, one can observe that CSR decompositions are an algebraic expression of turnpike phenomena occurring in tropical dynamical systems driven by one matrix. Namely, they express the fact that in such systems there are optimal trajectories (or walks) with a special structure: after a finite number of steps they arrive to a well-defined group of nodes called critical nodes, then dwell within that group of nodes, and then use a finite number of steps to reach the destination. The same phenomena will likely occur in inhomogenous products as well, but only under certain restrictive conditions. In particular, we can agree that all matrices constituting these inhomogeneous products have the same sets of critical nodes, and for a starter, we can consider the case where all these matrices have just one critical node. Under this and some other assumptions,  Shue et al.~\cite{Shue-98} found that products  $\Gamma(k)$ become tropical rank-$1$ matrices (i.e., tropical outer products) when $k$ is sufficiently big. Kennedy-Cochran-Patrick et al.~\cite{KCP-19} improved this result by giving a lower bound for $k$ to guarantee that $\Gamma(k)$ becomes a rank-$1$ matrix (i.e., a tropical outer product). In the present paper we show that the above results of~\cite{KCP-19,Shue-98} can be generalised further by introducing the {\em factor rank transient:} the length of the product after which the product is guaranteed to have a tropical factor rank not exceeding certain number, and by extending the concept of CSR decomposition to inhomogeneous products. Recall that tropical factor rank of a matrix $A$, studied together with many other concepts of rank in Akian et al.~\cite{AGG-09}, can be defined as follows: for a matrix $A\in\Rmax^{n\times m}$, the {\em tropical factor rank $r$} of $A$ is the smallest $r\in \mathbb{N}$ such that $A=U\otimes L$ where $U\in\Rmax^{n\times r}$ and $L\in\Rmax^{r\times m}$.
Note that the factor rank of $A$ is also equal to the minimum number of factor rank-$1$ matrices whose sum is equal to $A$, see~\cite{AGG-09}[Definition 7.1]. 

For wider reading, Hook~\cite{Hook17} shows that, by approximating the rank of the product in a min-plus setting, one can find and express the predominant structure in the associated digraph of the matrices forming the product. Hook has also looked at turnpike theory with respect to the max-plus linear systems in~\cite{Hook13}. In this paper he studies infinite length products, then uses a turnpike property to develop a factorisation of said matrix product. In terms of turnpikes, many results were obtained for them in the context of dynamic programming, in both discrete and continuous settings. Specifically, Kontorer and Yakovenko~\cite{Yakovenko92} used turnpike theory and Bellman equations to work with discrete optimal control problems. Following his work, Kolokoltsov and Maslov~\cite{IdempotentAnalysis} developed turnpike theory for discrete optimal control problems in the context of idempotent analysis and tropical mathematics.

The paper will proceed as follows. The first section will cover the necessary definitions and notation as well as a brief overview of~\cite{KCP-19} to give a more concrete background for the ensuing work. In~\cref{cpt:ambientcase} we look to generalise the work from~\cite{KCP-19} to a general case. For~\cref{cpt:nobound} we look at the cases where no lower bound can exist using counterexamples.

\section{Definitions and Notation}

\subsection{Weighted digraphs and tropical matrices}

This subsection presents some concepts and notation expressing the connection between tropical matrices and weighted digraphs. Monographs~\cite{MainBook,MPatWork} are our basic references for such definitions.

\begin{definition}[Weighted digraphs] \label{d:digraph}
A \emph{directed graph} \emph{(digraph)} is a~pair $(\nodes,\edges)$ where $\nodes$ is a~finite set of nodes and $\edges \subseteq \nodes \cross \nodes = \{(i,j)\colon i,j \in \nodes \}$ is the~set of edges, where $(i,j)$ is a~directed edge from node $i$ to node $j$.
	
A \emph{weighted digraph} is a~digraph with associated weights $w_{i,j} \in \Rmax$ for each edge $(i,j)$ in the~digraph.
	
A \emph{digraph associated with a~square matrix $A$} is a weighted digraph $\digr(A) = (\nodes_A, \edges_A)$ where the~set $\nodes_A$ has the same number of elements as the number of rows or columns in the~matrix $A$. The~set $\edges_A \subseteq \nodes_A \cross \nodes_A$ is the~set of edges in $\digr(A)$, where $(i,j)$ is an edge if and only if $a_{i,j}\neq -\infty$, and in this case the~weight of $(i,j)$ equals the~corresponding entry in the~matrix $A$, i.\,e. $w_{i,j} = a_{i,j} \in \Rmax$. 
\end{definition}

\begin{definition}[Walks, paths and weights] 
\label{d:walks}
A sequence of nodes $W=(i_0,\ldots,i_l)$ is called a~\emph{walk on a~weighted digraph} $\digr=(\nodes,\edges)$ if
	$(i_{s-1},i_s)\in \edges$ for each $s\colon 1\leq s \leq l$.
	This walk is a~\emph{cycle} if the start node $i_0$ and the~end node $i_l$ are the~same.
	It is a~\emph{path} if no two nodes in $i_0,\ldots, i_l$ are the~same. The~\emph{length} of $W$ is $l(W)=l$.\\  The~\emph{weight} of $W$ is defined as the~max-plus product (i.\,e., the~usual arithmetic sum) of the~weights of each edge $(i_{s-1},i_s)$ traversed throughout the walk, and it is denoted by $p_{\digr}(W)$. Note that a~sequence $W=(i_0)$ is also a~walk (without edges), and we assume that it has weight and length $0$.\\
	The \emph{mean weight} of $W$ is defined as the ratio $p_{\digr}(W)/l(W)$. 
\end{definition}

For a digraph, being strongly connected is a particularly useful property.

\begin{definition}[Strongly connected, irreducible, completely reducible]	
\label{d:scirred}
A digraph is \emph{strongly connected}, if for any two nodes $i$ and $j$ there exists a walk connecting $i$ to $j$. A square matrix is \emph{irreducible} if the graph associated with it in the sense of Definition~\ref{d:digraph} is strongly connected.

A digraph is called \emph{completely reducible}, if it consists of a number of strongly connected components, such that no two nodes of any two different components can be connected to each other by a walk.
\end{definition}

Note that, trivially, any strongly connected digraph is completely reducible. 

The following more refined notions are crucial in the study of ultimate periodicity of tropical matrix powers, and also for the present paper.

\begin{definition}[Cyclicity and cyclic classes] 
\label{d:cyclicity}
The~\emph{cyclicity} of a completely reducible digraph is the lowest common multiple of the highest common factors of the lengths of cycles within each strongly connected component. It will be denoted by $\gamma$. 

For two nodes $i,j\in\nodes$ we say that $i$ and $j$ are in the same~\emph{cyclic class} if there exists a walk of length modulo $\gamma$ connecting $i$ to $j$ or $j$ to $i$. This splits the set of nodes into $\gamma$ cyclic classes: $\cclass_0,\ldots, \cclass_{\gamma-1}$. The notation $\cclass_l\to_k \cclass_m$ means that some (and hence all) walks connecting nodes of $\cclass_l$ to nodes of $\cclass_m$ have lengths congruent to $k$ modulo $\gamma$.  The cyclic class containing $i$ will be also denoted by $[i]$.
\end{definition}

The correctness of the above definition of cyclic classes follows, for example, from \cite[Lemma 3.4.1]{BR:91}: in fact, every walk from $i$ to $j$
on $\digr$ has the same length modulo $\gamma$.

In tropical algebra, we often have to deal with two digraphs: 1) the digraph associated with $A$ and 2) the critical digraph of $A$. The latter digraph (being a subdigraph of the first) is defined below.

\begin{definition}[Maximum cycle mean and critical digraph]
For a square matrix $A$, the \emph{maximum cycle mean} of $\digr(A)$ denoted as $\lambda(A)$ (equivalently, the maximum cycle mean of $A$) is the biggest mean weight of all cycles of $\digr(A)$. 

A cycle in $\digr(A)$ is called \emph{critical} if its mean weight is equal to the maximum cycle mean (i.e., is maximal). 

The \emph{critical digraph} of $A$, denoted by $\crit(A)$, is the subdigraph of $\digr(A)$ whose node set $\critnodes$ and edge set $\critedges$ consist of all nodes and edges that belong to the critical cycles (i.e., that are critical). 
\end{definition}

Note that any critical digraph is completely reducible. As shown already in~\cite{CDQV-83,CDQV-85}, the cyclicity of critical digraph of $A$ is the ultimate period of the tropical matrix powers sequence $\{A^t\}_{t\geq 1}$, provided that $A$ is irreducible and $\lambda(A)=0$. See also Butkovi\v{c}~\cite{MainBook} and Sergeev~\cite{Ser-09} for more detailed analysis of the ultimate periodicity of this sequence.

Below we will use notation for walk sets and their maximal weights that is similar to that of Merlet~et~al.~\cite{MNS}.

\begin{definition}[Sets of walks]
\label{d:walkrep}
Let $\digr=(\nodes,\edges)$ be a weighted digraph and let $i,j\in\nodes$. The three sets $\walks{i}{j}{\digr}$, $\walkslen{i}{j}{k}{\digr}$ and $\walksnode{i}{j}{\digr}{\mathcal{N}}$, where $\mathcal{N} \subseteq \nodes$ is a subset of nodes, are defined as follows:
\begin{itemize}
    \item [] $\walks{i}{j}{\digr}$ is the set of walks over $\digr$ connecting $i$ to $j$;
    \item [] $\walkslen{i}{j}{k}{\digr}$ is the set of walks over $\digr$ of length $k$ connecting $i$ to $j$;
    \item [] $\walksnode{i}{j}{\digr}{\mathcal{N}}$ is the set of walks over $\digr$ connecting $i$ to $j$ that traverse at least one node of $\mathcal{N}$.
\end{itemize}
The supremum of the weights of walks in these sets will be denoted by $p(\mathcal{W})$.
\end{definition} 

\subsection{Main assumptions}
\label{ss:assumptions}

In this subsection, we set out the main assumptions about 
$\mathcal{X}$ and the matrices $A_{\alpha}$ that are drawn from this set and give some relevant definitions.

\begin{definition}[Geometrical equivalence]
Let the matrices $A$ and $B$ have their respective digraphs $\digr(A)=(\nodes_{A},\edges_{A})$ and $\digr(B)=(\nodes_{B},\edges_{B})$. We say that $A$ and $B$ are \emph{weakly geometrically equivalent} if $\nodes_{A}=\nodes_{B}$ and $\edges_{A}=\edges_{B}$, and they are \emph{strongly geometrically equivalent} if they are weakly geometrically equivalent and $\crit(A)=\crit(B)$.
\end{definition}

We cannot assume that the maximum cycle mean of each $A_{\alpha}\in\mathcal{X}$ is zero therefore we normalise each matrix to give the new set of matrices $\mathcal{Y}$, where 
\begin{align*}
    \mathcal{Y}=\{ A'_{\alpha}: A'_{\alpha}=A_{\alpha}\otimes\lambda^{-}(A_{\alpha})\;\forall A_{\alpha}\in\mathcal{X}\}.
\end{align*}

\begin{notation}[$A^{\sup}$ and $A^{\inf}$]
\label{not:asupinf}
\begin{itemize}
\item[] $A^{\sup}$: entrywise supremum of all matrices in $\mathcal{Y}$. In formula, $\Asup=\bigoplus_{\alpha\colon A_{\alpha}\in\mathcal{Y}} A_{\alpha}$. 
\item[] $\Ainf$: entrywise infimum of all matrices in $\mathcal{Y}$.
\end{itemize}
\end{notation}

Note that the concept of $\Asup$ has been used before for various purposes. In~\cite{MA-AHP}, Gursoy, Mason and Sergeev use the same definition to develop a common subeigenvector for the entire semigroup of matrices used to create $\Asup$, which is a technique we will use later on. In~\cite{Gursoy11}, Gursoy and Mason use $\Asup$, and $\lambda(\Asup)$ to develop bounds for the max-eigenvalues over a set of matrices.

\begin{asssumption}{A}{}
\label{as:irred}
Any matrix $A_{\alpha}\in \mathcal{X}$ is irreducible.
\end{asssumption}

\begin{asssumption}{B}{}
\label{as:geom}
Any two matrices $A_{\alpha}, A_{\beta}\in\mathcal{X}$ are strongly geometrically equivalent to each other and to 
$\Asup$, which has all entries in $\Rmax$.
\end{asssumption}

\begin{notation}
\label{not:digrcrit}
The common associated digraph of the matrices from $\mathcal{X}$ will be denoted by $\digr(\mathcal{X})=(\nodes,\edges)$, and the common critical digraph by $\crit(\mathcal{X})=(\critnodes,\critedges)$. In general, this critical digraph has $m\geq 1$ strongly connected components, denoted by $\crit_{\nu}$, for $\nu=1,\ldots,m.$

\end{notation}

\begin{asssumption}{C}{}
\label{as:infgeom}
Any matrix $A_{\alpha}\in\mathcal{X}$ is weakly geometrically equivalent to $\Ainf$. In other words, for each $(i,j) \in \edges$, we have  $(\Ainf)_{ij} \neq - \infty$. 
\end{asssumption}

\begin{asssumption}{D}{1}
\label{as:eigensup}
For the matrix $A^{sup}$, we have $\lambda(A^{sup})=0$.
\end{asssumption}

The first three assumptions come from the previous works by Shue et al.~\cite{Shue-98} and Kennedy-Cochran-Patrick~et~al.~\cite{KCP-19}: however, we will no longer assume that the critical graph consists just of one loop.

The final assumption below is inspired by the visualisation scaling studied in Sergeev~et~al~\cite{SSB}, see also~\cite{Ser-FP} and references therein for more background on this scaling. 

\begin{definition}[Visualisation]
Matrix $B$ is called a \emph{visualisation} of $A$ if there exists a diagonal matrix $X=\diag(x)$, with entries $X_{ii}=x_i$ on the diagonal and $X_{ij}=-\infty$ off the diagonal (i.e., if $i\neq j$), such that $B=X^{-1}AX$ and $B$ satisfies the following conditions: $B_{ij}=\lambda(B)$ for $(i,j)\in\critedges(B)$ and $B_{ij}\leq\lambda(B)$ for $(i,j)\notin\critedges(B)$.
\end{definition}

Once $\lambda(A)\neq -\infty$, a visualisation of $A$ always exists and, moreover, vectors $x$ providing a visualisation by means of diagonal matrix scaling $A\mapsto X^{-1}AX$ are precisely the tropical subeigenvectors of $A$, i.e., vectors satisfying $Ax\leq \lambda(A)x$. Using this information we have the following lemma.

\begin{lemma} \label{l:visscaling}
Suppose that the vector $x$ satisfies $\Asup x \leq x$. Then $x$ provides a simultaneous visualisation for all matrices of $\mathcal{X}$ (and $\mathcal{Y}$).
\end{lemma}
\begin{proof}
Let $x$ be the vector that satisfies $\Asup x \leq x$. By construction, $\Asup$ is the supremum matrix of all the normalised generators in $\mathcal{X}$. Therefore for these normalised generators $A_{\alpha}$, $A_{\alpha} \leq \Asup$. Hence the vector $x$ also satisfies $A_{\alpha} x\leq x$ and it can be used to visualise $A_{\alpha}$. As this applies for all $\alpha$ then they can be simultaneously visualised. As $\mathcal{Y}$ is the set of normalised matrices from $\mathcal{X}$ then the same applies to any matrix from $\mathcal{Y}$ as well.
\end{proof}

This is referred to as the set of matrices having a \emph{common visualisation}, therefore, without loss of generality we assume that we have performed this common visualisation on all of the matrices in $\mathcal{X}$ (and $\mathcal{Y}$) to give the final core assumption.

\begin{asssumption}{D}{2}
\label{as:vis}
For all $A_{\alpha}\in\mathcal{Y}$, we have $(A_{\alpha})_{ij}=0$ and $(\Asup)_{ij}=0$ for $(i,j)\in\critedges$, and $(A_{\alpha})_{ij}\leq 0$ and 
$(\Asup)_{ij}\leq 0$ for $(i,j)\notin\critedges$.
\end{asssumption}

From now on we will use Assumption~\ref{as:vis} instead of Assumption~\ref{as:eigensup} without loss of generality.

\subsection{Extension to inhomogeneous products}

Recall now that we have a set of matrices $\mathcal{Y}$,
from which we can select matrices in arbitrary sequence.

\begin{definition} \label{d:word}
The~\emph{word} associated with the matrix product $\Gamma(k)$ is the string of characters $i$ from $A_{i}\in\mathcal{Y}$ that make up said $\Gamma(k)$.
\end{definition}

Let us also introduce the trellis digraph associated with a matrix product 
$\Gamma(k)=A_1\otimes A_2\otimes\ldots\otimes A_k$ (as in~\cite{KCP-19}, inspired by Viterbi algorithm).

\begin{definition} \label{d:trellis}
The \emph{trellis digraph} $\trellis(\Gamma(k)) = (\trellisnodes, \trellisedges)$ associated with the~product $\Gamma(k)=A_1\otimes A_2\otimes\ldots\otimes A_k$ is the~digraph with the~set of nodes $\trellisnodes$ and the~set of edges $\trellisedges$, where:
	\begin{itemize}
		\item[(1)] $\trellisnodes$ consists of $k+1$ copies of
		$\nodes$ which are denoted $\nodes_0,\ldots, \nodes_k$, and the~nodes in $\nodes_l$ for each $0\leq l\leq k$ are denoted by $1:l,\ldots, n:l$;
		\item[(2)] $\trellisedges$ is defined by the following rules:
		\begin{itemize}
			\item[a)] there are edges only between $\nodes_l$ and $\nodes_{l+1}$ for each $l$,
			\item[b)] we have $(i:(l-1),j:l)\in \trellisedges$ if and only if $(i,j)$ is an edge of $\digr(\mathcal{Y})$, and the~weight of that edge is $(A_l)_{i,j}$.
		\end{itemize}
	\end{itemize}
	The weight of a walk $W$ on $\trellis(\Gamma(k))$ is denoted by $p_{\trellis}(W)$.
\end{definition}

Below we will need to use 1) walks that start at one side of the trellis and end at an intermediate node, 2) walks that start at an intermediate node and end at the other side of the trellis, 3) walks that connect one side of the trellis to the other. More formally, we give the following definition. 

\begin{definition} \label{d:walks2}
Consider a~trellis digraph $\trellis(\Gamma(k))$.
	
	By an~\emph{initial walk} connecting $i$ to $j$ on  $\trellis(\Gamma(k))$ we mean a~walk on $\trellis(\Gamma(k))$ connecting node $i:0$ to $j:m$, where  $0\leq m\leq k$. 
	
	By a~\emph{final walk} connecting $i$ to $j$ on  $\trellis(\Gamma(k))$ we mean a~walk on $\trellis(\Gamma(k))$  connecting node $i:l$ to $j:k$, where $0\leq l\leq k$. 
	
	A~\emph{full walk} connecting $i$ to $j$ on $\trellis(\Gamma(k))$ is a~walk on $\trellis(\Gamma(k))$ connecting node $i:0$ to $j:k$.
\end{definition}

We will mostly work with the following sets of walks on $\trellis$.

\begin{notation}[Walk sets on $\trellis(\Gamma(k))$]
\label{not:trelliswalks}
\begin{itemize}
    \item[] $\walkslen{i}{j}{k}{\trellis,\full},$ $\walkslen{i}{j}{l}{\trellis,\init}$ and $\walkslen{i}{j}{l}{\trellis,\final}:$  set of full walks (of length $k$), and sets of initial and final walks of length $l$ on $\trellis$ connecting $i$ to $j$. 
    \item[] $\walkslennode{i}{j}{k}{\trellis,\full}{\critnodes},$ $\walkslennode{i}{j}{l}{\trellis,\init}{\critnodes}$ and $\walkslennode{i}{j}{l}{\trellis,\final}{\critnodes}:$  set of full walks (of length $k$), and sets of initial and final walks of length $l$ on $\trellis$ traversing a critical node and connecting $i$ to $j$;
    \item[] $\walks{i}{\critnodes\|}{\trellis,\init}$: set of initial walks connecting $i$ to a node in $\critnodes$ so that this node of $\critnodes$ is the only node of $\critnodes$ that is visited by the walk and it is visited only once;
    \item[] $\walks{\|\critnodes}{j}{\trellis,\final}$: set of final walks connecting a node in $\critnodes$ to $j$ so that this node of $\critnodes$ is the only node of $\critnodes$ that is visited by the walk and it is visited only once.
\item[] $i\to_{\trellis} j:$ this denotes the situation where $i:0$ can be connected to $j:k$ on $\trellis$ by a full walk.
\end{itemize}
\end{notation}

Recall that $p(\mathcal{W})$ denotes the optimal weight of a walk in a set of walks $\mathcal{W}$.
The \emph{optimal walk interpretation} of entries of $\Gamma(k)$
in terms of walks on $\trellis=\trellis(\Gamma(k))$ is now apparent:
\begin{equation} \label{eq:walkrep}
    \Gamma(k)_{i,j} = p\left(\walkslen{i}{j}{k}{\trellis,\full}\right).
\end{equation}

 We will also need special notation for the optimal weights of walks in the sets $\walks{i}{\critnodes\|}{\trellis,\init}$ and $\walks{\|\critnodes}{j}{\trellis,\final}$ introduced above.
 
\begin{notation}[Optimal weights of walks on $\trellis(\Gamma(k))$]
\label{not:w*v*}
\begin{itemize}
    \item[] $w^{\ast}_{i,\critnodes}=p(\walks{i}{\critnodes\|}{\trellis,\init})$ : the maximal weight of walks in $\walks{i}{\critnodes\|}{\trellis,\init}$,
    \item[] $v^{\ast}_{\critnodes,j}=p(\walks{\|\critnodes}{j}{\trellis,\final})$ : the maximal weight of walks in  $\walks{\|\critnodes}{j}{\trellis,\final}$.
\end{itemize}
\end{notation}

The following notation is for optimal values of various optimisation problems involving paths and walks on $\digr(\Asup)$, $\digr(\Ainf)$, which will be used in our factor rank bounds.

\begin{notation}[Optimal weights of walks on $\digr(\Asup)$ and $\digr(\Ainf)$]
\label{not:weightssupinf}
\begin{itemize}
    \item[] $\alpha_{i,\critnodes}$ : the weight of the optimal path on $\digr(\Asup)$ connecting node $i$ to a node in $\critnodes$;
    \item[] $\beta_{\critnodes,j}$ : the weight of the optimal path on $\digr(\Asup)$ connecting a node in $\critnodes$ to node $j$;
    \item[] $\gamma_{i,j}$ : the weight of the optimal path on $\digr(\Asup)$ connecting node $i$ to node $j$ without traversing any node in $\critnodes$.
    \item[] $w_{i,\critnodes}$ : the weight of the optimal path on $\digr(\Ainf)$ connecting node $i$ to a node in $\critnodes$;
    \item[] $v_{\critnodes,j}$ : the weight of the optimal path on $\digr(\Ainf)$ connecting a node in $\critnodes$ to node $j$;
    \item[] $u^{k}_{i,j}$ : the weight of the optimal walk on $\digr(\Ainf)$ of length $k$ connecting node $i$ to node $j$.
\end{itemize}
\end{notation}

We remark by saying that the Kleene star, which is explored in~\cite{MainBook} and is defined as $(A)^{\ast}=I \oplus A \oplus A^{2} \oplus \ldots$, of $\Asup$ can be used to find the values of $\alpha_{i,\critnodes}$ and $\beta_{\critnodes,j}$. Similarly the Kleene star of $\Ainf$ can be used to find $w_{i,\critnodes}$ and $v_{\critnodes,j}$. Let us end this section with the following observation, which follows from the geometric equivalence (Assumptions~\ref{as:geom} and~\ref{as:infgeom})

\begin{lemma}
\label{l:elemequiv}
The following are equivalent:
{\rm (i)} $i\to_{\trellis} j$;
{\rm (ii)} $(\Gamma(k))_{i,j}>\tropzero$;
{\rm (iii)} $u^k_{i,j}>\tropzero$.
\end{lemma}

\section{CSR products}
\label{s:CSR}

In this section we introduce CSR decomposition of inhomogeneous products and study its properties. We will give the two definitions of the CSR decomposition of $\Gamma(k)$ and prove their equivalence. However in order to do that we require another definition.

\begin{definition} \label{d:ultimatep}
Let the matrix $A$ have cyclicity $\gamma$. The \emph{threshold of ultimate periodicity} of powers of $A$, is a bound $T(A)$ such that $\forall k \geq T(A)$, $A^{k}=A^{k+\gamma}$.
\end{definition}

This threshold is required to develop the CSR decomposition for $\Gamma(k)$ as seen in the following definitions.

\begin{definition}[CSR-1] 
\label{d:CSR1}
Let $\Gamma(k)=A_{1}\otimes\ldots\otimes A_{k}$ be a matrix product of length $k$. Define $C$, $S$ and $R$ as follows:
\begin{itemize}
    \item[] $S$ is the matrix associated with the 
    critical graph, i.e.
    \begin{equation}
        S=(s_{i,j})=\begin{cases} 0 & \text{if $(i,j)\in\critedges$} \ \\
        \tropzero & \text{otherwise.}
        \end{cases}
    \end{equation}
   \item[] Let $\gamma$ be the cyclicity of critical graph, and $t$ be a big enough number, such that $t\gamma\geq T(S)$, where $T(S)$ is the threshold of ultimate periodicity of (the powers of) $S$. 
    \item[] $C$ and $R$ are defined by the following formulae:
    \begin{equation*}
    C=\Gamma(k)\otimes S^{(t+1)\gamma-k\modd\gamma}, \quad R=S^{(t+1)\gamma-k\modd\gamma}\otimes\Gamma(k).    
    \end{equation*}
         \item[] The product of $C$, $S^{k\modd{\gamma}}$ and $R$ will be denoted by $CS^{k\modd{\gamma}}R[\Gamma(k)]$. We say that $\Gamma(k)$ is \emph{CSR} if  $CS^{k\modd{\gamma}}R[\Gamma(k)]$ is equal to $\Gamma(k)$. 
\end{itemize}
\end{definition}

For completeness we must also state that for any matrix in $A\in\Rmax^{n\times n}$, $A^{0}=I$, where $I$ is the tropical identity matrix, i.e. $I=\diag(0)$. In the next definition, we prefer to define CSR terms corresponding to the components of the critical graph.

\begin{definition}[CSR-2] 
\label{d:CSR2}
Let $\Gamma(k)=A_{1}\otimes\ldots\otimes A_{k}$ be a matrix product of length $k$, and let $\crit_{\nu}$, for $\nu=1,\ldots,m$ be the components of $\crit(\mathcal{Y})$. For each $\nu=1,\ldots,m$ define $C_{\nu}$, $S_{\nu}$ and $R_{\nu}$ as follows: 
\begin{itemize}
    \item[] $S_{\nu}\in\Rmax^{n \times n}$ is the matrix associated with the s.c.c. $\crit_{\nu}$ of the 
    critical graph, i.e.,
    \begin{equation}
        S_{\nu}=(s_{i,j})=\begin{cases} 0 & \text{if $(i,j)\in\crit_{\nu}$}, \ \\
        \tropzero & \text{otherwise.}
        \end{cases}
    \end{equation}
   \item[] Let $\gamma_{\nu}$ be the cyclicity of critical component, and $t_{\nu}$ be a big enough number, such that $t_{\nu}\gamma_{\nu}\geq T(S_{\nu})$, where $T(S_{\nu})$ is the threshold of ultimate periodicity of (the powers of) $S_{\nu}$. 
    \item[] $C_{\nu}$ and $R_{\nu}$ are defined by the following formulae:
    \begin{equation*}
    C_{\nu}=\Gamma(k)\otimes S_{\nu}^{(t_{\nu}+1)\gamma_{\nu}-k\modd{\gamma_{\nu}}}, \quad R_{\nu}=S_{\nu}^{(t_{\nu}+1)\gamma_{\nu}-k\modd{\gamma_{\nu}}}\otimes\Gamma(k).    
    \end{equation*}
         \item[] The product of $C_{\nu}$, $S_{\nu}^{k\modd{\gamma_{\nu}}}$ and $R_{\nu}$ will be denoted by $C_{\nu}S_{\nu}^{k\modd{\gamma_{\nu}}}R_{\nu}[\Gamma(k)]$. We say that $\Gamma(k)$ is \emph{CSR} if  
         $$\Gamma(k)=\bigoplus_{\nu=1}^m C_{\nu} S_{\nu}^{k\modd{\gamma_{\nu}}}R_{\nu}[\Gamma(k)].$$
\end{itemize}
\end{definition}

Using the definitions given above, we can write out the 
CSR terms more explicitly:
\begin{equation*}
\begin{split}
    CS^{k\modd{\gamma}}R[\Gamma(k)] & = \Gamma(k)\otimes S^{(t+1)\gamma -k\modd{\gamma}} \otimes S^{k\modd{\gamma}}\otimes S^{(t+1)\gamma -k\modd{\gamma}} \otimes \Gamma(k)\\
    & = \Gamma(k) \otimes S^{2(t+1)\gamma-k\modd\gamma} \otimes \Gamma(k),\\
     C_{\nu}S_{\nu}^{k\modd{\gamma_{\nu}}}R_{\nu}[\Gamma(k)] &= \Gamma(k) \otimes S_{\nu}^{2(t_{\nu}+1)\gamma_{\nu}-k\modd{\gamma_{\nu}}} \otimes \Gamma(k),
\end{split}
\end{equation*}
Since the powers of $S$ are ultimately periodic with period 
$\gamma$ and the powers of $S_{\nu}$ are ultimately periodic with period $\gamma_{\nu}$, and since also we have $t\gamma\geq T(S)$ and $t_{\nu}\gamma_{\nu}\geq T(S_{\nu})$, we can reduce the exponents of $S$ and $S_{\nu}$ to $(t+1)\gamma-k\modd{\gamma}$ and $(t_{\nu}+1)\gamma_{\nu}-k\modd{\gamma_{\nu}}$, respectively, and thus 
\begin{equation}
\label{e:CSRdirect}    
 \begin{split}
& CS^{k\modd{\gamma}}R[\Gamma(k)]=\Gamma(k)\otimes S^v\otimes \Gamma(k),\quad  C_{\nu}S_{\nu}^{k\modd{\gamma_{\nu}}}R_{\nu}[\Gamma(k)]=
  \Gamma(k)\otimes S_{\nu}^{v_{\nu}}\otimes \Gamma(k),\\
& \text{for}\  v=(t+1)\gamma-k\modd\gamma,\  v_{\nu}=(t_{\nu}+1)\gamma_{\nu}-k\modd{\gamma_{\nu}},\ t\gamma\geq T(S),\ t_{\nu}\gamma_{\nu}\geq T(S_{\nu}).
\end{split}
\end{equation}
Below we will also need the following elementary observation.

\begin{lemma}
\label{l:elem}
Let $v=(t+1)\gamma-k\modd\gamma$, where $t\gamma\geq T(S)$. 
Then, for any $\nu$, we can find $t_{\nu}$ such that  $v=(t_{\nu}+1)\gamma_{\nu}-k\modd{\gamma_{\nu}}$ and $t_{\nu}\gamma_{\nu}\geq T(S_{\nu})$. 
\end{lemma} 
 \begin{proof} 
 The existence of $t_{\nu}$ such that 
 $v=(t_{\nu}+1)\gamma_{\nu}-k\modd{\gamma_{\nu}}$
 follows since $\gamma$ is a multiple of $\gamma_{\nu}$, and then we also have 
 $t_{\nu}\gamma_{\nu}\geq t\gamma\geq T(S)\geq T(S_{\nu})$.
 \end{proof}
This lemma allows us to also write 
\begin{equation}
    \label{e:CSRdirect2}
  C_{\nu}S_{\nu}^{k\modd{\gamma_{\nu}}}R_{\nu}[\Gamma(k)]=
  \Gamma(k)\otimes S_{\nu}^v\otimes \Gamma(k),   
\end{equation}
with $v$ as in~\eqref{e:CSRdirect}.

\begin{proposition}~\label{p:definitionequiv}
$\Gamma(k)$ is CSR by Definition~\ref{d:CSR1} if and only if 
it is CSR by Definition~\ref{d:CSR2}.
\end{proposition}
\begin{proof}
We need to show that
\begin{equation}
\label{e:CSRexp}
   CS^{k\modd{\gamma}}R[\Gamma(k)]=\bigoplus_{\nu=1}^m C_{\nu} S_{\nu}^{k\modd{\gamma_{\nu}}}R_{\nu}[\Gamma(k)]
\end{equation}
for arbitrary $k$.
Using~\eqref{e:CSRdirect} and~\eqref{e:CSRdirect2} we can rewrite this equivalently as 
\begin{equation}
\label{e:gammakequal}
    \Gamma(k) \otimes S^{(t+1)\gamma-k\modd{\gamma}} \otimes \Gamma(k) = \Gamma(k) \otimes \left( \bigoplus_{\nu=1}^{m} S_{\nu}^{(t+1)\gamma-k\modd{\gamma}} \right) \otimes \Gamma(k)
\end{equation}
with $t\gamma\geq T(S)$. To obtain this equality, observe that $S=\bigoplus_{\nu=1}^{m}S_{\nu}$, and as $S_{\nu_1}\otimes S_{\nu_2}=-\infty$ for any $\nu_1$ and $\nu_2$ we can raise both sides to the same power to give us $S^{t}=\bigoplus_{\nu=1}^{m}S_{\nu}^{t}$ for any $t$.
This shows~\eqref{e:gammakequal}, and the claim follows.
\end{proof}
For a similar reason, we also have the following identities:
\begin{equation}
\label{e:CRdecomp}
\begin{split}
C&=\bigoplus_{\nu=1}^m C_\nu,\qquad R=\bigoplus_{\nu=1}^m
R_\nu,\\
C\otimes S^{k\modd{\gamma}}&=\bigoplus_{\nu=1}^m C_{\nu}\otimes S_{\nu}^{k\modd{\gamma_{\nu}}},\quad S^{k\modd\gamma}\otimes R=\bigoplus_{\nu=1}^m S_{\nu}^{k\modd{\gamma_{\nu}}}\otimes R_{\nu}.
\end{split}    
\end{equation}

To give an optimal walk interpretation of CSR, we will need to define the trellis graph corresponding to these terms, by modifying Definition~\ref{d:trellis}.

\begin{definition}[Symmetric extension of the trellis graph]
\label{d:trellisext}
Let $v=(t+1)\gamma -k\modd{\gamma}$, where $t$ be a large enough number such that $t\gamma\geq T(S)$.\\ 
Define $\trellis'(\Gamma(k))$ as the digraph $\trellis'=(\trellisnodes',\trellisedges')$ with the set of nodes $\trellisnodes'$ and edges $\trellisedges'$, such that:
\begin{itemize}
    \item[(1)] $\trellisnodes'$ consists of $2k+v+1$ copies of $\nodes$ which are denoted $\nodes_{0},\ldots,\nodes_{2k+v}$ and the nodes for $\nodes_{l}$ for each $0\leq l \leq 2k+v$ are denoted by $1:l,\ldots,n:l$;
    \item[(2)] $\trellisedges'$ is defined by the following rules:
    \begin{itemize}
        \item[a)] there are edges only between $\nodes_{l}$ and $\nodes_{l+1}$,
        \item[b)] for $1\leq l \leq k$ we have $(i:l-1,j:l)\in \trellisedges'$ if and only if $(i,j)\in\edges(\mathcal{Y})$ and the weight of the edge is $(A_{l})_{i,j}$,
        \item[c)] for $k+v+1\leq l \leq 2k+v$ we have $(i:l-1,j:l)\in \trellisedges'$ if and only if $(i,j)\in\edges(\mathcal{Y})$ and the weight of the edge is $(A_{l-k-v})_{i,j}$,
        \item[d)] for $k < l < k+v+1$ we have $(i:l-1,j:l)\in \trellisedges'$ if and only if $(i,j)\in\crit(\mathcal{Y})$ and the weight of the edge is $0$.
    \end{itemize}
\end{itemize}
The weight of a walk on $\trellis'(\Gamma(k))$ is denoted by $p_{\trellis'}(W)$.
\end{definition}
The following optimal walk interpretation of CSR terms on $\trellis'$ is now obvious.

\begin{lemma}[CSR and optimal walks] \label{r:CSRrep}
The following identities hold for all $i,j$
\begin{equation}
\begin{split}
\label{e:optwalkCSR}
(CS^{k\modd{\gamma}}R[\Gamma(k)])_{i,j}
& = p\left(\walkslen{i}{j}{2k+v}{\trellis',\full}\right),\\
(C_{\nu}S_{\nu}^{k\modd{\gamma_{\nu}}}R_{\nu}[\Gamma(k)])_{i,j}
& = p\left(\walkslennode{i}{j}{2k+v}{\trellis',\full}{\critnodes^{\nu}}\right),
\end{split}
\end{equation}
where $v=(t+1)\gamma-k\modd{\gamma},$ with $t\gamma\geq T(S)$.
\end{lemma}
\begin{proof}
With~\eqref{e:CSRdirect}, the first identity follows from the optimal walk interpretation of $\Gamma(k)\otimes S^v\otimes \Gamma(k)$, and the second identity follows from~\eqref{e:CSRdirect2} and the optimal walk interpretation of 
$\Gamma(k)\otimes S_{\nu}^v\otimes\Gamma(k)$.
\end{proof}

In what follows, we mostly work with Definition~\ref{d:CSR2}, but we can switch between the equivalent definitions if we find it convenient.

We now present a useful lemma that shows equality for columns of $C_{\nu}$ and rows of $R_{\nu}$ with indices in the same cyclic class.

\begin{lemma} \label{l:cyclicequality}
For any $i$ and for any two nodes $x$ and $y$ in the same cyclic class of the critical component $\crit_{\nu}$ we have
\begin{equation}
    (C_{\nu})_{i,x} = (C_{\nu})_{i,y} \quad \text{ and } \quad (R_{\nu})_{x,i} = (R_{\nu})_{y,i}
\end{equation}
\end{lemma}
\begin{proof}
We prove the lemma for columns, as the case of the rows is similar. 

For any $i,j$, denote $(C_{\nu})_{i,j}$ by $c_{i,j}$. From the definition of $C_{\nu}$, it follows that $c_{i,x}$ is the weight of an optimal walk in $\walkslennode{i}{j}
{k+(t_{\nu}+1)\gamma_{\nu}-k\modd{\gamma_{\nu}}}{\trellis',\init}{\critnodes^{\nu}}$ where $t_{\nu}\gamma_{\nu}\geq T(S_{\nu})$, and such walk consists of two parts. The first part is a full walk on $\trellis$ connecting $i$ to the critical subgraph at some node $s$.  The second part is a walk over the critical subgraph of length $(t_{\nu}+1)\gamma_{\nu} -k\modd{\gamma_{\nu}}$ connecting $s$ to $x$ with weight zero. As the length of the second walk is greater than $T(S_{\nu})$, a walk connecting $s$ to $x$ exists if and only if $[s]\to_{-k\modd{\gamma_{\nu}}}[x]$. If a full walk connecting $i$ to $[s]$ on $\trellis$ exists then, for arbitrary $x,y$ in the same cyclic class,  $c_{i,x}$ and $c_{i,y}$ are both equal to the optimal weight of all walks connecting $i$ to $[s]$ on $\trellis$,  where $[s]\to_{-k\modd{\gamma_{\nu}}}[x]$, otherwise both $c_{i,x}$ and $c_{i,y}$ are equal to $-\infty$. This shows that $c_{i,x}=c_{i,y}$. 

The case of rows of $R_{\nu}$ is considered similarly, but instead of initial walks one has to use final walks on $\trellis'$.
\end{proof}

We can use this to prove the same property for $C$ and $R$ of Definition~\ref{d:CSR1}.

\begin{corollary} \label{c:cyclicequalityfull}
For any $i$ and for any two nodes $x$ and $y$ in the same critical component and the same cyclic class of said critical component, we have
\begin{equation}
    C_{i,x} = C_{i,y} \quad \text{ and } \quad R_{x,i} = R_{y,i}
\end{equation}
\end{corollary}
\begin{proof}
We will prove only the first identity, as the proof of the second identity is similar.
Let $x,y$ belong to the same component $\crit_{\mu}$ of $\crit(\mathcal{Y})$, and let them belong to the same cyclic class of that component. By Lemma~\ref{l:cyclicequality}
we have $(C_{\mu})_{i,x}=(C_{\mu})_{i,y}$, and we also have $(C_{\nu})_{i,x}=(C_{\nu})_{i,y}=\tropzero$ for any $\nu\neq\mu$.
Using these identities and~\eqref{e:CRdecomp},  we have
\begin{align*}
     C_{i,x}=
     \left(\bigoplus_{\nu=1}^{m} C_{\nu}\right)_{i,x} = (C_{\mu})_{i,x}=(C_{\mu})_{i,y}=
     \left(\bigoplus_{\nu=1}^{m} C_{\nu}\right)_{i,y}
     =C_{i,y}.
     \end{align*}
\end{proof}

The next theorem explains why CSR is useful for inhomogeneous products. Note that in the proof of it we use the CSR structure rather than the $\Gamma(k)\otimes S^v\otimes \Gamma(k)$ representation that was used above.

\begin{theorem}
\label{t:CSRrank}
 The factor rank of each $C_{\nu}S_{\nu}^{k\modd{\gamma_{\nu}}}R_{\nu}[\Gamma(k)]$ is no more than $\gamma_{\nu}$, for $\nu=1,\ldots,m$, and the factor rank of $CS^{k\modd{\gamma}}R[\Gamma(k)]$ is no more than $\sum_{\nu=1}^m \gamma_{\nu}$.
\end{theorem}
\begin{proof}


For each $\nu=1,\ldots,m,$ take all the nodes from $\subcrit_{\nu}$ and order them into cyclic classes $\cclass^\nu_{0},\ldots,\cclass^\nu_{\gamma_{\nu}-1}$. 
Take two columns with indices $x,y\in\cclass^{\nu}_{i}$ from the matrix $C_{\nu}$. As they are in the same cyclic class, by Lemma~\ref{l:cyclicequality} the columns are equal to each other. This means that we can take a column representing a single node from each cyclic class and since there are $\gamma_{\nu}$ distinct classes then there will be $\gamma_{\nu}$ distinct columns of $C_{\nu}$. The same also holds for any two rows of $R_{\nu}$: if the row indices are in the same cyclic class, then the rows are equal, so that we have $\gamma_{\nu}$ distinct rows.

Let us now check that the same holds for $S_{\nu}^{k\modd{\gamma_{\nu}}}\otimes R_{\nu}$. By the construction of $S_{\nu}^{k\modd{\gamma_{\nu}}}$ we know that if $(S_{\nu}^{k\modd{ \gamma_{\nu}}})_{ij}\neq 0$ then $[i] \to_{k\modd{\gamma_{\nu}}}[j]$. Therefore $$(S_{\nu}^{k\modd{ \gamma_{\nu}}}\otimes R_{\nu})_{i,\cdot}=\bigoplus_{j\in N_c} (S_{\nu}^{k\modd{ \gamma_{\nu}}})_{ij}\otimes  (R_{\nu})_{j,\cdot}=\bigoplus_{j\colon [i]\to_{k \modd{ \gamma_{\nu}}} [j]} (S_{\nu}^{k\modd{ \gamma_{\nu}}})_{ij}\otimes  (R_{\nu})_{j,\cdot}=(R_{\nu})_{j,\cdot}.$$ This means that for a row $i$ such that $[i]\to_{k \modd{ \gamma_{\nu}}} [j]$ we have $(S_{\nu}^{k\modd{\gamma_{\nu}}}\otimes R_{\nu})_{i,\cdot}=(R_{\nu})_{j,\cdot}$ and all such rows of $S_{\nu}^{k\modd{\gamma_{\nu}}}\otimes R_{\nu}$ are equal to each other.

Our next aim is to define, for each $\nu$, matrices $C'_{\nu}$ and $R'_{\nu}$ with $\gamma_{\nu}$ rows and $\gamma_{\nu}$ columns, such that 
$C_{\nu}S_{\nu}^{k\modd{\gamma_{\nu}}}R_{\nu}[\Gamma(k)]=C'_{\nu}\otimes R'_{\nu}$. To form matrix $C'_{\nu}$, we select a node of $\crit_{\nu}$ from each cyclic class $\cclass^\nu_{0},\ldots,\cclass^\nu_{\gamma_{\nu}-1}$ and define the column of $C'_{\nu}$ whose index is the number of this node to be the column of $C_{\nu}$
with the same index. The rest of the columns of $C'_{\nu}$ are set to $-\infty$. To form matrix $R'_{\nu}$, we use the same selected nodes, but this time (instead of taking columns of $C_{\nu}$ and making them columns of $C'_{\nu}$) we take the rows from $S_{\nu}^{k\modd{\gamma_{\nu}}}\otimes R_{\nu}$ whose indices are the numbers of selected nodes and make them rows of $R'_{\nu}$. The rest of the rows of $R'_{\nu}$ are set to $-\infty$. 
Since the rows of $C_{\nu}$ with indices in the same cyclic class are equal to each other and the same is true about the rows of $S_{\nu}^{k\modd{\gamma_{\nu}}}\otimes R_{\nu}$, we have $C_{\nu}S_{\nu}^{k\modd{\gamma_{\nu}}}R_{\nu}[\Gamma(k)]=C'_{\nu}\otimes R'_{\nu}$, thus the factor rank of any of these terms is no more than $\gamma_{\nu}$.

We next form the matrices $C'=\bigoplus_{\nu=1}^m C'_{\nu}$ and $R'=\bigoplus_{\nu=1}^m R'_{\nu}$. Obviously, $C'_{\nu_1}\otimes R'_{\nu_2}=-\infty$ for $\nu_1\neq \nu_2$ and therefore 
\begin{align*}
    C'\otimes R' = \bigoplus_{\nu=1}^{m} C'_{\nu} \otimes R'_{\nu} = \bigoplus_{\nu=1}^m C_{\nu} S_{\nu}^{k\modd{\gamma_{\nu}}} R_{\nu}[\Gamma(k)]=CS^{k\modd{\gamma}}R[\Gamma(k)].
\end{align*}
Finally, as $C'$ and, respectively, $R'$ have 
$\sum_{\nu=1}^m \gamma_{\nu}$ columns with finite entries and, respectively, rows with finite entries with the same indices, $CS^{k\modd{\gamma}}R[\Gamma(k)]=C'\otimes R'$ has factor rank at most $\sum_{\nu=1}^{m} \gamma_{\nu}$.


\end{proof}

\begin{corollary}
\label{c:CSRrank} 
If $\Gamma(k)$ is CSR, then its rank is no more than $\sum_{\nu=1}^m \gamma_{\nu}$.
\end{corollary}

Let us also prove the following results that are similar to \cite[Corollary 3.7]{SerSch}. 

\begin{proposition}
\label{p:CSRcrit}
For each $\nu=1,\ldots,m$
\begin{align*}
    (C_{\nu}\otimes S_{\nu}^{k\modd{\gamma_{\nu}}}\otimes R_{\nu})_{\cdot,j} = (C_{\nu}\otimes S_{\nu}^{k\modd{\gamma_{\nu}}})_{\cdot,j} \quad \text{for} \quad j\in \critnodes^{\nu} \\ (C_{\nu}\otimes S_{\nu}^{k\modd{\gamma_{\nu}}}\otimes R_{\nu})_{i,\cdot} = (S_{\nu}^{k\modd{\gamma_{\nu}}}\otimes R_{\nu})_{i,\cdot} \quad \text{for} \quad i\in\critnodes^{\nu}.
\end{align*}
\end{proposition}
\begin{proof}
As the proofs are very similar for both statements we will only prove the first and omit the proof for the second statement. We begin by observing that $$(C_{\nu}\otimes S_{\nu}^{k\modd{\gamma_{\nu}}})_{i,j}=p\left(\walkslen{i}{j}{k+t_{\nu}\gamma_{\nu}}{\trellis', \init}\right)
,$$
where we used the definitions of $C_{\nu}$ and $S_{\nu}$ and the identity
$S_{\nu}^{(t_{\nu}+1)\gamma_{\nu}}=S_{\nu}^{t_{\nu}\gamma_{\nu}}$ (since $t_{\nu}\gamma_{\nu}\geq T(S_{\nu})$).
Here it is convenient to choose $t_{\nu}$ that satisfies $(t_{\nu}+1)\gamma_{\nu}-k\modd{\gamma_{\nu}}=(t+1)\gamma-k\modd\gamma$, with 
$t$ used in the definition of $\trellis'$.  With this choice $t_{\nu}\gamma_{\nu}\leq t\gamma$.

 Using~\eqref{e:optwalkCSR}, all we need to show is that $p\left(\walkslennode{i}{j}{2k+v}{\trellis',\full}{\critnodes^{\nu}}\right) = p\left(\walkslen{i}{j}{k+t_{\nu}\gamma_{\nu}}{\trellis',\init}\right)$, where $v=(t+1)\gamma-k\modd\gamma$. We will achieve this by proving these two inequalities:
\begin{equation} \label{eq:CSRcolumn}
\begin{split}
    p\left(\walkslennode{i}{j}{2k+v}{\trellis',\full}{\critnodes^{\nu}}\right) & \geq p\left(\walkslen{i}{j}{k+t_{\nu}\gamma_{\nu}}{\trellis',\init}\right),\\  
 p\left(\walkslennode{i}{j}{2k+v}{\trellis',\full}{\critnodes^{\nu}}\right) & \leq p\left(\walkslen{i}{j}{k+t_{\nu}\gamma_{\nu}}{\trellis',\init}\right)
    \end{split}
\end{equation}
To prove the first inequality of~\eqref{eq:CSRcolumn} we first consider $\walkslen{i}{j'}{k+t_{\nu}\gamma_{\nu}}{\trellis',\init},$ where $j'\in [j]$. Optimal walk in any of these sets can be decomposed into 
1) an optimal full walk on $\trellis$ connecting $i$ to a node of $[j]$, and 2) a walk of weight $0$ and length $t_{\nu}\gamma_{\nu}$ on $\crit_{\nu}$ connecting that node of $[j]$ to $j'$, whose existence follows since $t_{\nu}\gamma_{\nu}\geq T(S_{\nu}).$ This decomposition implies that the weights of all these optimal walks are equal. One of them, denote it by $W_1$ can be concatenated with a walk $W_2$ on $\crit_{\nu}$ of length $k-k\modd{\gamma_{\nu}}+\gamma$ and ending in $j$. We see that $p(W_1W_2)=p(W_1)$ and $W_1W_2\in \walkslennode{i}{j}{2k+v}{\trellis',\full}{\critnodes^{\nu}}.$


To prove the second inequality of~\eqref{eq:CSRcolumn} we take a
walk in $\walkslennode{i}{j}{2k+v}{\trellis',\full}{\critnodes^{\nu}}$ and decompose it into 1) a walk in $\walkslen{i}{j'}{k+t_{\nu}\gamma_{\nu}}{\trellis',\init},$ where $j'\in [j]$, 2) a walk in $\walkslen{j'}{j}{k-k\modd{\gamma_{\nu}}+\gamma_{\nu}}{\trellis',\final}$. 
The weight of the first walk is bounded by $p\left(\walkslen{i}{j}{k+t_{\nu}\gamma_{\nu}}{\trellis',\init}\right)$, and the weight of the second walk is bounded by $0$, thus the second inequality also holds.
\end{proof}

\begin{corollary} \label{c:CSRcritfull}
For CSR as defined in Definition~\ref{d:CSR1} we have,
\begin{align*}
    (C\otimes S^{k\modd{\gamma}}\otimes R)_{\cdot,j} = (C\otimes S^{k\modd{\gamma}})_{\cdot,j} \quad \text{for} \quad j\in \critnodes \\ (C\otimes S^{k\modd{\gamma}}\otimes R)_{i,\cdot} = (S^{k\modd{\gamma}}\otimes R)_{i,\cdot} \quad \text{for} \quad i\in\critnodes.
\end{align*}
\end{corollary}
\begin{proof}
The proofs for both statements are similar so we will only prove the first one.

Let $j\in\critnodes$. As all nodes from $\critnodes$ can be sorted into $\critnodes^{\nu}$ for some $\nu=1,\ldots,m$, assume without loss of generality that $j\in\critnodes^{\mu}$.

Taking the right-hand side of the first statement and using~\eqref{e:CRdecomp}, we have
\begin{equation*}
    (C\otimes S^{k\modd{\gamma}})_{\cdot,j} = \left(\bigoplus_{\nu=1}^{m} C_{\nu}\otimes S_{\nu}^{k\modd{\gamma_{\nu}}}\right)_{\cdot,j}.
\end{equation*}
By Definition~\ref{d:CSR2}, if $j\in\critnodes^{\mu}$ then for all $\nu\neq\mu$, $(C_{\nu}\otimes S_{\nu}^{k\modd{\gamma_{\nu}}})_{\cdot,j} = -\infty$. Therefore, for every $\nu$, $(C_{\nu}\otimes S_{\nu}^{k\modd{\gamma_{\nu}}})_{\cdot,j}$ will be dominated by $(C_{\mu}\otimes S_{\mu}^{k\modd{\gamma_{\mu}}})_{\cdot,j}$. Hence,
\begin{equation} \label{e:CSmu}
    \left(\bigoplus_{\nu=1}^{m} C_{\nu}\otimes S_{\nu}^{k\modd{\gamma_{\nu}}}\right)_{\cdot,j} = (C_{\mu}\otimes S_{\mu}^{k\modd{\gamma_{\mu}}})_{\cdot,j}.
\end{equation}

Turning our attention to the left-hand side of the first statement, by~\eqref{e:CRdecomp} we get
\begin{equation*}
    (C\otimes S^{k\modd{\gamma}}\otimes R)_{\cdot,j} = \left(\bigoplus_{\nu=1}^{m} C_{\nu}\otimes S_{\nu}^{k\modd{\gamma_{\nu}}}\otimes R_{\nu}\right)_{\cdot,j}.
\end{equation*}
Now we must show that, for $j\in\critnodes^{\mu}$ and for all $\nu$, $(C_{\nu}\otimes S_{\nu}{k\modd{\gamma_{\nu}}}\otimes R_{\nu})_{\cdot,j}\leq (C_{\mu}\otimes S_{\mu}^{k\modd{\gamma_{\mu}}}\otimes R_{\mu})_{\cdot,j}$. By~\eqref{e:optwalkCSR} this is the same as saying 
\begin{align*}
    p\left( \walkslennode{i}{j}{2k+v}{\trellis',\full}{\critnodes^{\nu}} \right) \leq  p\left( \walkslennode{i}{j}{2k+v}{\trellis',\full}{\critnodes^{\mu}} \right)
\end{align*}
for some arbitrary node $i$. Let $W$ be the walk of length $2k+v$ connecting $i$ to $j$ that traverses $\critnodes^{\nu}$, such that $p(W)= p\left( \walkslennode{i}{j}{2k+v}{\trellis',\full}{\critnodes^{\nu}} \right)$. As $j\in\critnodes^{\mu}$ then $W$ is also a walk of length $2k+v$ connecting $i$ to $j$ that traverses $\critnodes^{\mu}$, hence $W\in\walkslennode{i}{j}{2k+v}{\trellis',\full}{\critnodes^{\mu}}$ and the inequality holds.

Therefore, as with the right-hand side, we have
\begin{equation} \label{e:CSRmu}
    \left(\bigoplus_{\nu=1}^{m} C_{\nu}\otimes S_{\nu}^{k\modd{\gamma_{\nu}}}\otimes R_{\nu}\right)_{\cdot,j} = (C_{\mu}\otimes S_{\mu}^{k\modd{\gamma_{\mu}}}\otimes R_{\mu})_{\cdot,j}.
\end{equation}
Finally the first statement of Proposition~\ref{p:CSRcrit} gives us equality between~\eqref{e:CSmu} and~\eqref{e:CSRmu}. As $j$ was chosen arbitrarily, this holds for any $j\in\critnodes$ and the result follows.
\end{proof}

\section{General results} \label{cpt:genresults}

This section presents some results that hold for general inhomogeneous products satisfying the assumptions set out in Section~\ref{ss:assumptions}. Before we proceed, let us introduce the following piece of notation, inspired by the weak CSR expansion of Merlet~et~al.~\cite{MNS}:
\begin{notation}[$B^{\sup}$ and $\lambda_*$]
Denote
\begin{equation*}
(B^{\sup})_{i,j}=
\begin{cases}
\tropzero, & \text{if $i\in\critnodes$ or $j\in\critnodes$},\\
(\Asup)_{i,j}, &\text{otherwise}
\end{cases}
\end{equation*}
and by $\lambda_{\ast}$ the maximum cycle mean of $B^{\sup}$.
\end{notation}

We remark that the the metric matrix, given in~\cite{MainBook} and defined as $A^{+}=A\oplus A^{2} \oplus\ldots$, of $B^{\sup}$ is useful in calculating all the entries of $\gamma_{i,j}$ simultaneously.

\begin{notation}[$q$]
\label{not:q} 
We will denote by $q$ the number of critical nodes, i.e., $q=|\critnodes|$.
\end{notation}

The following results generalize~\cite[Lemmas 3.1-3.2]{KCP-19} for initial and final walks to the case of a general critical subgraph.
Observe that, under Assumptions~\ref{as:geom} and~\ref{as:vis},
we have $\lambda_{\ast}<0$, so that the bounds in the following lemmas make sense. Recall the sets of walks $\walks{i}{\critnodes\|}{\trellis,\init}$ and $\walks{\|\critnodes}{j}{\trellis,\final}$ introduced in Notation~\ref{not:trelliswalks}.

\begin{lemma} \label{l:initial}
Let $W_{i,\critnodes}$ be an optimal walk in $\walks{i}{\critnodes\|}{\trellis,\init}$, so that 
$p(W_{i,\critnodes})=w^*_{i,\critnodes}$. Then we have the following bound on the length of $W_{i,\critnodes}$:
\begin{equation}
\label{e:initialbound}
    l(W_{i,\critnodes}) \leq 
    \begin{cases}
    n-q, &\text{if $\lambda_\ast=\tropzero$},\\
    \frac{ w_{i,\critnodes}^{\ast} - \alpha_{i,\critnodes}}{\lambda_{\ast}} + (n-q), & \text{if $\lambda_*>\tropzero$}
    \end{cases}
\end{equation}
\end{lemma}
\begin{proof}
If $\lambda_*=\tropzero$, then any walk in 
$\walks{i}{\critnodes\|}{\trellis,\init}$ has to be a path, and its length is bounded by $n-q$.
Now let $\lambda_{\ast}>\tropzero$.
As $\lambda_{\ast} < 0$, the weight of the walk $W_{i,\critnodes}$ connecting $i$ to a node in $\critnodes$ is less than or equal to that of a path $P_{i,\critnodes}$ on $\digr(\Asup)$ connecting $i$ to a node in $\critnodes$ plus the remaining length multiplied by $\lambda_{\ast}$. The remaining length is bounded from above by $n-q$, since all intermediate nodes in $W_{i,\critnodes}$ are non-critical. Hence
\begin{equation*}
    p_{\trellis}(W_{i,\critnodes}) \leq p_{\sup}(P_{i,\critnodes}) + (l(W_{i,\critnodes}) - (n-q))\lambda_{\ast}.
\end{equation*}
We can bound $p_{\sup}(P_{i,\critnodes}) \leq \alpha_{i,\critnodes}$, so
\begin{equation} \label{intial-a}
   p_{\trellis}(W_{i,\critnodes}) \leq \alpha_{i,\critnodes} + (l(W_{i,\critnodes}) - (n-q))\lambda_{\ast}.
\end{equation}
Now assuming for contradiction that  $l(W_{i,\critnodes}) > \frac{ w_{i,\critnodes}^{\ast} - \alpha_{i,\critnodes} }{\lambda_{\ast}} + (n-q)$ . This is equivalent to 
\begin{equation} \label{intial-b}
    \alpha_{i,\critnodes} + (l(W_{i,\critnodes}) - (n-q))\lambda_{\ast} < w^{\ast}_{i,\critnodes}.
\end{equation}
In combining~\eqref{intial-a} and~\eqref{intial-b} we get $ p_{\trellis}(W_{i,\critnodes}) < w^{\ast}_{i,\critnodes}$ meaning that $W_{i,\critnodes}$ is not optimal, a contradiction. So we know that for for any $l\in \critnodes$ 
\begin{equation*}
     l(W_{i,\critnodes}) \leq \frac{w_{i,\critnodes}^{\ast} - \alpha_{i,\critnodes} }{\lambda_{\ast}} + (n-q).
\end{equation*}
The proof is complete.
\end{proof}

\begin{lemma} \label{l:final}
Let $W_{\critnodes,j}$ be an optimal walk in $\walks{\|\critnodes}{j}{\trellis,\final}$, so that 
$p(W_{\critnodes,j})=v^*_{\critnodes,j}$. Then we have the following bound on the length of $W_{\critnodes,j}$:
\begin{equation}
\label{e:finalbound}
    l(W_{\critnodes,j}) \leq
    \begin{cases}
    n-q, & \text{if $\lambda_*=\tropzero$},\\ 
    \frac{v_{\critnodes,j}^{\ast} - \beta_{\critnodes,j}}{\lambda_{\ast}} + (n-q), &\text{if $\lambda^*>\tropzero$}.
\end{cases}
\end{equation}
\end{lemma}
As the proof of this lemma is analogous to the proof of Lemma~\ref{l:initial} it is omitted. 
Also, we can observe that $n-q$ is the limit of the expressions on the right-hand side of~\eqref{e:initialbound} and~\eqref{e:finalbound} as $\lambda_*\to\tropzero$, hence we will not consider this case separately in the rest of the paper.

The following result is a generalised form of~\cite[Lemma 3.4]{KCP-19} which uses a nominal weight $\omega$.

\begin{lemma} \label{l:avoiding}
If $\gamma_{i,j}=\tropzero$, then any full walk connecting $i$ to $j$ on $\trellis(\Gamma(k))$ traverses a node in $\critnodes$.\\
If $\gamma_{i,j}>\tropzero$, let
\begin{equation} \label{full-l}
    k > \frac{\omega-\gamma_{i,j}}{\lambda_{\ast}}+(n-q)
\end{equation}
for some $\omega \in \mathbb{R}$. Then any full walk $W$ connecting $i$ to $j$ on $\trellis(\Gamma(k))$ that does not go through any node $l\in \critnodes$ has weight smaller than $\omega$.
\end{lemma}
\begin{proof}
In the case when $\gamma_{i,j}=\tropzero$, the claim follows by the definition of $\gamma_{i,j}$ and by the geometric equivalence between $\Asup$ and the matrices from $\mathcal{Y}$. So we assume that $\gamma_{i,j}> \tropzero$. Any walk $W$ that does not traverse any node in $\critnodes$ can be decomposed into a path $P$ connecting $i$ to $j$  avoiding $\critnodes$ and a number of cycles. Hence we have the following bound:
\begin{equation*}
    p_{\trellis}(W) \leq p_{\sup}(P) + (k-(n-q))\lambda_{\ast}.
\end{equation*}
We can further bound $p_{\sup}(P) \leq \gamma_{i,j}$ so
\begin{equation} \label{full-a}
    p_{\trellis}(W) \leq \gamma_{i,j} + (k-(n-q))\lambda_{\ast}.
\end{equation}
Now~\eqref{full-l} can be rewritten as
\begin{equation} \label{full-b}
    \gamma_{i,j} + (k-(n-q))\lambda_{\ast} < \omega.
\end{equation}
By combining~\eqref{full-a} with~\eqref{full-b} we have $p_{\trellis}(W) < \omega$, which completes the proof.
\end{proof}

Using this bound we can obtain a bound after which the CSR term becomes a valid upper bound for $\Gamma(k)$.

\begin{theorem} \label{th:weakcsr}
If $\gamma_{i,j}=\tropzero$ then $\Gamma(k) \leq CS^{k\modd{\gamma}}R[\Gamma(k)]$.\\
If $\gamma_{i,j}>\tropzero$, let  
\begin{equation} \label{eq:weakcsr}
    k > \max_{i,j\colon i\to_\trellis j,\gamma_{i,j}>\tropzero} \left( \frac{\Gamma(k)_{i,j}-\gamma_{i,j}}{\lambda_\ast} + (n-q) \right).
\end{equation}
Then $\Gamma(k) \leq CS^{k\modd{\gamma}}R[\Gamma(k)]$.
\end{theorem}
\begin{proof}
If $i\not\to_\trellis j$, then $(\Gamma(k))_{i,j}=-\infty$. In this case, obviously, $\Gamma(k)_{i,j}\leq (CS^{k\modd{\gamma}}R[\Gamma(k)])_{i,j}$.  


If $i\to_\trellis j$, then $(\Gamma(k))_{i,j}\neq \tropzero$. Let $W^{\ast}$ be the optimal walk of length $k$ on $\trellis(\Gamma(k))$ connecting $i$ to $j$ with weight $\Gamma(k)_{i,j}$. If $k$ is greater than the bound~\eqref{eq:weakcsr} then, by Lemma~\ref{l:avoiding}, for the walk to have weight equal to $\Gamma(k)_{i,j}$, it must traverse at least one node in $\critnodes$, and the same is true when $\gamma_{i,j}=\tropzero$. Hence this walk belongs to the set $\walkslennode{i}{j}{k}{\trellis}{\critnodes}$ and further $\Gamma(k)_{i,j}=p(W^{\ast}) \leq p\left( \walkslennode{i}{j}{k}{\trellis}{\critnodes} \right)$.

Let $f\in\critnodes$ be the first critical node in the first critical s.c.c $\crit_{\nu}$, with cyclicity $\gamma_{\nu}$, that $W^{\ast}$ traverses. We can split the walk into $W^{\ast}=W_{1}W_{3}$ where $W_{1}$ is a walk connecting $i$ to $f$ of length $r$ and $W_{3}$ is a walk connecting $f$ to $j$ of length $k-r$. We have $p(W^{\ast})=p(W_{1})+p(W_{3})$. 

Let $\trellis'$ be the trellis extension for the matrix product $CS^{k\modd{\gamma}}R[\Gamma(k)]$ with length $2k + v$ where $v=(t+1)\gamma-k\modd{\gamma}$ as described in Definition~\ref{d:trellisext}.

We now introduce the new walk $W'=W_{1}W_{2}W_{3}$ on $\trellis'$. Here $W_{1}$ and $W_{3}$ are the subwalks from $W^{\ast}$ introduced before, where $W_1$ is viewed as an initial walk on $\trellis'$ and $W_3$ as a final walk on $\trellis'$, and $W_{2}$ is a closed walk of length $k+v$ that starts and ends at $f$. Since $k+v\equiv 0\modd{\gamma_{\nu}}$ and $k+v\geq T(S)\geq T(S_{\nu})$, this closed walk exists and can be entirely made up of edges from $\crit_{\nu}$. This means the walk $W'$ is of length $2k+v$ and it traverses the set of nodes $\critnodes^{\nu}$ therefore $W' \in \walkslennode{i}{j}{2k+v}{\trellis'}{\critnodes^{\nu}}$.

As $W_{2}$ is made entirely from critical edges, we have $p(W_{2})=0$ and $p(W^{\ast})=p(W')\leq p\left( \walkslennode{i}{j}{2k+v}{\trellis'}{\critnodes^{\nu}} \right)$, and using~\eqref{e:optwalkCSR2} gives us
\begin{equation*}
    \Gamma(k)_{i,j} = p(W^{\ast}) \leq (C_{\nu}S^{k\modd{\gamma_{\nu}}}_{\nu}R_{\nu}[\Gamma(k)])_{i,j} \leq (CS^{k\modd{\gamma}}R[\Gamma(k)])_{i,j},
\end{equation*}
where the last inequality is due to Proposition~\ref{p:definitionequiv}. The claim follows.

\end{proof}

This bound is implicit, as it requires $\Gamma(k)$ to be calculated in order to generate the transient. However, we can use $A^{\inf}$ and $u_{i,j}$ to develop an explicit bound.

\begin{corollary} \label{c:weakcsrinf}
Let 
\begin{equation} \label{eq:weakcsrinf}
    k > \max_{i,j\colon i\to_\trellis j, \gamma_{i,j}>\tropzero} \left( \frac{u^{k}_{i,j}-\gamma_{i,j}}{\lambda_{\ast}} + (n-q) \right).
\end{equation}
Then $\Gamma(k) \leq CS^{k\modd{\gamma}}R[\Gamma(k)]$.
\end{corollary}
\begin{proof}
By Lemma~\ref{l:elemequiv}, $i\to_\trellis j$ is equivalent to 
$u_{i,j}^k>\tropzero$, so maximum 
in~\eqref{eq:weakcsrinf} is taken over $i,j$ for which $u^k_{i,j}$ and $\gamma_{i,j}$ are finite. We also have
$u_{i,j}^k\leq(\Gamma(k))_{i,j}$ by the definition of $A^{\inf}$.




Further, as $\lambda_{\ast} < 0$, then any $k$ that satisfies~\eqref{eq:weakcsrinf} will also satisfy~\eqref{eq:weakcsr}. The claim now follows from Theorem~\ref{th:weakcsr}.
\end{proof}

\begin{remark}
All the results in this section do not require proper visualisation scaling on the matrices from $\mathcal{Y}$, but we need $\lambda_{\ast}<0$ and we require all critical edges to have weight zero in all matrices of $\mathcal{Y}$.
\end{remark}

\section{The case where CSR works} \label{cpt:ambientcase}

In the case when $\crit(\mathcal{X})$ is just one loop,
Kennedy-Cochrane-Patrick et al.~\cite{KCP-19} established a bound on the lengths of inhomogeneous products, after which these products are of tropical factor rank $1$.  In this section we extend this result to the case when 
$\digr(\mathcal{X})$ and $\crit(\mathcal{X})$ satisfy the following assumption, in addition to the assumptions that were set out in Section~\ref{ss:assumptions}.


\begin{asssumption}{P}{0} \label{as:ambientgamma}
$\crit(\mathcal{X})$ is strongly connected and its cyclicity $\gamma$ is equal to the cyclicity of $\digr(\mathcal{X})$.
\end{asssumption}

The equality between cyclicities means that the associated digraph $\digr(\mathcal{X})$ has the same number of cyclic classes $\gamma$ as $\crit(\mathcal{X})$. 

\begin{notation}
The cyclic classes of $\digr(\mathcal{X})$ are denoted by $\cclass'_0,\ldots,\cclass'_{\gamma-1}$.\\
For a node $i\in\nodes ,$ the cyclic class of this node with respect to $\digr(\mathcal{X})$ will be denoted by $[i]'$.  
\end{notation}

For a node $i\in\critnodes$, we will use both $[i]$ (the cyclic class with respect to $\crit(\mathcal{X})$) and $[i]'$ (the cyclic class with respect to $\digr(\mathcal{X})$), and an obvious inclusion relation between them: $[i]\subseteq [i]'$.

One of the ideas is to combine Lemmas~\ref{l:initial} and~\ref{l:final} together with Schwarz's bound. To define this bound, following~\cite{MNS}, we first introduce {\em Wielandt's number}
\begin{align*}
\wiel(n) = \begin{cases}
(n-1)^{2}+1 & \text{if} \; n\geq 1, \\
0 & \text{if} \; n=0,
\end{cases}
\end{align*} 
and then {\em Schwarz's number}
\begin{equation*}
\sch(\gamma,n)= \gamma\wiel\left(\left\lfloor\frac{n}{\gamma}\right\rfloor\right)+n\modd\gamma.
\end{equation*}

Let us now prove the following lemma.

\begin{lemma} \label{l:ambientcyclicitygamma}
Let 
\begin{equation}
    k \geq   \frac{ w_{i,\critnodes}^{\ast} - \alpha_{i,\critnodes} }{\lambda_{q\ast}} + (n-q) + \sch(\gamma,q)+  \frac{ v_{\critnodes,j}^{\ast} - \beta_{\critnodes,j} }{\lambda_{q\ast}} + (n-q).
\end{equation}
Then
\begin{itemize}
\item[(i)] If $[i]'\not{\to}_k[j]'$ then there are no full walks connecting $i$ to $j$ on $\trellisgammak$ (i.e., $i\not\to_{\trellis} j$).
\item [(ii)] If $[i]'\to_k[j]'$, then there is a full walk $W$ connecting $i$ to $j$ on $\trellisgammak$ and going through a critical node, and we have
$p_{\trellis}(W) = w_{i,\critnodes}^{\ast} + v_{\critnodes,j}^{\ast}$ if $W$ is optimal.
\end{itemize}
\end{lemma}
\begin{proof}
The property $[i]'\not\to_k [j]'$ implies that there is no full walk $W$ connecting $i$ to $j$
on $\trellisgammak$.

In the case $[i]'\to_k [j]'$, we construct a walk $W'=W_{i,\critnodes}W_c W_{\critnodes,j}$ of length $k$, 
where $W_{i,\critnodes}$ be an optimal walk in $\walks{i}{\critnodes\|}{\trellis,\init}$ (see Lemma~\ref{l:initial}) , $W_{\critnodes,j}$ be an optimal walk in $\walks{\|\critnodes}{j}{\trellis,\final}$ (see Lemma~\ref{l:final}),
and $W_c$ is a walk that connects the end of $W_{i,\critnodes}$ to the beginning of $W_{\critnodes,j}$ and such that all edges of $W_c$ are critical (the existence of such $W_c$ is yet to be proved). Without loss of generality set $[i]'=\cclass'_{0}$ and $[j]'=\cclass'_{p_{3}}$: the cyclic classes of $\digr(\mathcal{X})$ to which $i$ and $j$ belong. Let $x$ be the final node of $W_{i,\critnodes}$ and let $y$ be the first node of $W_{\critnodes,j}$. Set $[x]'=\cclass'_{p_{1}}$ and 
$[y]'= \cclass'_{p_{2}}$.

By \cite[Lemma 3.4.1.iv]{BR:91} $l(W_{i,\critnodes})\equiv p_{1}\modd{\gamma}$, $l(W_{\critnodes,j})\equiv (p_{3}-p_{2})\modd{\gamma}$. Hence the congruence of the walk $W_c$ to be inserted is $(p_{3}-p_{1}-(p_{3}-p_{2}))\modd{\gamma} \equiv (p_{2}-p_{1})\modd{ \gamma}$. As the cyclicity of the critical subgraph is the same as that of the digraph, the cyclic classes of the critical subgraph are $\cclass_{0},\ldots, \cclass_{\gamma-1}$ and we can assume that the numbering is such that $\cclass_{0} \subseteq \cclass'_{0}$,\ldots, $\cclass_{\gamma-1}\subseteq \cclass'_{\gamma-1}$. Then $x\in \cclass_{p_{1}}$ and $y \in \cclass_{p_{2}}$ and by \cite[Lemma 3.4.1.iv]{BR:91} there exists a walk on the critical subgraph of length congruent to $(p_{2}-p_{1})\modd{\gamma}$. Moreover, all walks connecting $x$ to $y$ have such length and by Schwarz's bound if $k-l(W_{i,\critnodes})-l(W_{\critnodes,j}) \geq \sch(\gamma,q)$ then there is a walk of length equal to $l(W')-l(W_{i,\critnodes})-l(W_{\critnodes,j})$. According to Lemmas~\ref{l:initial} and~\ref{l:final} $l(W_{i,\critnodes}) \leq \frac{ w_{i,\critnodes}^{\ast} - \alpha_{i,\critnodes} }{\lambda_{\ast}} + (n-q)$ , $l(W_{\critnodes,j}) \leq \frac{ v_{\critnodes,j}^{\ast} - \beta_{\critnodes,j} }{\lambda_{\ast}} + (n-q)$, therefore $k$ is a sufficient length for $k-l(W_{i,\critnodes})-l(W_{\critnodes,j})$ 
to satisfy Schwarz's bound, so a walk of the form $W'=W_{i,\critnodes} W_c  W_{\critnodes,j}$ exists and $p(W')=w^{\ast}_{i,\critnodes}+v^{\ast}_{\critnodes,j}$.

Let now $W$ be an optimal full walk connecting $i$ to $j$ on $\trellis$ that passes through $\critnodes$ at least once. As it passes through the critical nodes then the walk can be decomposed into $W=\tilde{W}_{i,\critnodes} \tilde{W_c} \tilde{W}_{\critnodes,j}$ where $\tilde{W}_{i,\critnodes}$ is a walk in $\walks{i}{\critnodes\|}{\trellis,\init}$, and  $\tilde{W}_{\critnodes,j}$ is a  walk in $\walks{\|\critnodes}{j}{\trellis,\final}$, and  
$\tilde{W_c}$ connects the end of  $\tilde{W}_{i,\critnodes}$ to the beginning of $\tilde{W}_{\critnodes,j}$ on $\trellis(\Gamma(k))$.
We then have  $p_{\trellis}(\tilde{W}_{i,\critnodes})\leq p_{\trellis}(W_{i,\critnodes})$ and $p_{\trellis}(\tilde{W}_{\critnodes,j}) \leq p_{\trellis}(W_{\critnodes,j})$ and also $p_{\trellis}(\tilde{W_c}) \leq p(W_c) = 0$. Since $W$ is optimal then all of these inequalities hold with equality, 
and $p_{\trellis}(W) = w_{i,\critnodes}^{\ast}+v_{\critnodes,j}^{\ast}$, as claimed.
\end{proof}


\begin{remark}
\label{r:turnpike}
It follows from the proof that, under the conditions of this lemma
and in the case $[i]\to_k[j]$, there is an optimal full walk connecting $i$ to $j$ on $\trellis_{\Gamma(k)}$ and traversing a critical node that can be decomposed as $W = W_{i,\critnodes}W_cW_{\critnodes,j}$, where $W_{i,\critnodes}$ is an optimal walk in $\walks{i}{\critnodes\|}{\trellis,\init}$  
and $W_{\critnodes,j}$ is an optimal walk
in $\walks{\|\critnodes}{j}{\trellis,\final}$, and $W_c$ consists of edges solely in the critical subgraph.
If semigroup's generators are also strictly visualised in the sense of~\cite{SSB}, then any such optimal full walk has to be of this form. 
\end{remark}

Lemma~\ref{l:ambientcyclicitygamma} gives us the first part of the final bound for the case. In order to be able to use this lemma we must ensure that the walk must traverse $\critnodes$ hence we can use Lemma~\ref{l:avoiding} in conjunction with Lemma~\ref{l:ambientcyclicitygamma} to give us the following theorem.

\begin{theorem} \label{th:ambientgamma}
Denote $u_{i,\critnodes,j}^{\ast}= w_{i\critnodes}^{\ast}+v_{\critnodes,j}^{\ast}$. Let
\begin{align} \label{eq:ambientgammabound1}
    k \geq \max  \left( \frac{ u_{i,\critnodes,j}^{\ast} - \alpha_{i,\critnodes}  - \beta_{\critnodes,j} }{\lambda_{\ast}} + 2(n-q) + 
    \sch(\gamma,q),\ 
    \frac{u_{i,\critnodes,j}^{\ast}-\gamma_{i,j}}{\lambda_{\ast}}+(n-q+1)\right)
\end{align}
if $\gamma_{i,j}>\tropzero$ or just 
\begin{align} \label{eq:ambientgammabound2}
    k \geq \frac{ u_{i,\critnodes,j}^{\ast} - \alpha_{i,\critnodes}  - \beta_{\critnodes,j} }{\lambda_{\ast}} + 2(n-q) + 
    \sch(\gamma,q),
\end{align}
if $\gamma_{i,j}=\tropzero$, for some $i,j \in \nodes$. Then  
\begin{itemize}
\item[{\rm (i)}] If $[i]'\not\to_k [j]'$ then $\Gamma(k)_{i,j}=-\infty$,
\item[{\rm (ii)}] If $[i]'\to_k [j]'$ then $\Gamma(k)_{i,j}= u_{i,\critnodes,j}^{\ast}=w^{\ast}_{i,\critnodes} + v^{\ast}_{\critnodes,j}$.
\end{itemize}
\end{theorem}
\begin{proof}
We only need to prove the second part. By Lemma~\ref{l:avoiding} and taking $\omega = w_{i,\critnodes}^{\ast}+v_{\critnodes,j}^{\ast}$, if 
\begin{equation*}
    k > \frac{w_{i,\critnodes}^{\ast}+v_{\critnodes,j}^{\ast}-\gamma_{i,j}}{\lambda_{q\ast}}+(n-q)
\end{equation*}
then any walk on $\trellis(\Gamma(k))$ that does not traverse the nodes in $\critnodes$ will have weight smaller than $w_{i,\critnodes}^{\ast}+v_{\critnodes,j}^{\ast}$, or such walk will not exist if $\gamma_{i,j}=\tropzero$. Using Lemma~\ref{l:ambientcyclicitygamma}, if
\begin{align*}
     k \geq   \frac{ w_{i,\critnodes}^{\ast} - \alpha_{i,\critnodes} }{\lambda_{q\ast}} + (n-q) + \sch(\gamma,q) +  \frac{ v_{\critnodes,j}^{\ast} - \beta_{\critnodes,j} }{\lambda_{q\ast}} + (n-q)
\end{align*}
and $[i]'\to_k [j]'$ then the weight of any optimal full walk on $\trellis(\Gamma(k))$ connecting $i$ to $j$ and traversing a critical node will be equal to $w_{i,\critnodes}^{\ast}+v_{\critnodes,j}^{\ast}$. If $\gamma_{i,j}=\tropzero$, $[i]'\to_k [j]'$ and the above inequality holds, or if $\gamma_{i,j}>\tropzero$, $k$ satisfies both inequalities and $[i]\to_k [j]$, then any optimal full walk  traverses nodes in $\critnodes$ and has weight
\begin{align*}
    \Gamma(k)_{i,j} = w_{i,\critnodes}^{\ast}+v_{\critnodes,j}^{\ast}.
\end{align*}
\end{proof}

Our next aim is to rewrite Theorem~\ref{th:ambientgamma} in a CSR form, and we first want to look at the optimal walk representation of $w_{i,\critnodes}^{\ast}$ and $v_{\critnodes,j}^{\ast}$. This leads to the following lemma.

\begin{lemma} \label{l:intrep}
 We have 
 \begin{equation}
w_{i,\critnodes}^{\ast}=p(\walkslen{i}{\critnodes}{k}{\trellis,\full}),\quad  v_{\critnodes,j}^{\ast}=p(\walkslen{\critnodes}{j}{k}{\trellis,\full}).
\end{equation}

\end{lemma}

\begin{proof}
We will prove only the first of these two equalities, as the second one can be proved in a similar way.  

Let $W_{i,\critnodes}$ be an optimal walk
in $\walks{i}{\critnodes\|}{\trellis,\init}$, with weight $w_{i,\critnodes}^{\ast}$.  
We are required to prove that
\begin{equation}
\label{e:ptoprove}
    p\left(\walks{i}{\critnodes\|}{\trellis,\init} \right) =  
    p\left(\walkslen{i}{\critnodes}{k}{\trellis,\full}\right),
\end{equation}
where on the right we have the set of full walks connecting $i$ to a critical node on $\trellis(\Gamma(k))$.
We split~\eqref{e:ptoprove} into two inequalities,
\begin{equation} \label{eq:maxrep}
  p\left(\walks{i}{\critnodes\|}{\trellis,\init} \right) \leq   
    p\left(\walkslen{i}{\critnodes}{k}{\trellis,\full}\right), \quad  p\left(\walks{i}{\critnodes\|}{\trellis,\init} \right) \geq   
    p\left(\walkslen{i}{\critnodes}{k}{\trellis,\full}\right)
\end{equation}

For the first inequality in~\eqref{eq:maxrep}, observe that 
we can concatenate $W_{i,\critnodes}$ with a walk $V$ on the critical graph which has length $l(V)=k-l(W_{i,\critnodes})$. 
The resulting walk $W_{i,\critnodes}V$ belongs to $\walkslen{i}{\critnodes}{k}{\trellis,\full}$
and has weight $w_{i,\critnodes}^*$, which proves the first inequality. For the second inequality, take an optimal walk $W^*\in \walkslen{i}{\critnodes}{k}{\trellis,\full}$,
whose weight is
$p(\walkslen{i}{\critnodes}{k}{\trellis,\full})$. By observing the first occurrence of a critical node in this walk, we represent $W^*=WV$, where $W\in\walks{i}{\critnodes\|}{\trellis,\init}$. We then have 
$p(W^*)=p(W)+p(V)\leq p(W)\leq w_{i,\critnodes}^{\ast}$ proving the second inequality. Combining both inequalities gives the equality~\eqref{e:ptoprove} and finishes the proof of $w_{i,\critnodes}^{\ast}=p(\walkslen{i}{\critnodes}{k}{\trellis,\full})$. The second part of the claim is proved similarly.
\end{proof}

\begin{remark}
\label{r:intreplm}
In the previous lemma, the length of the walks on the right-hand side does not have to be restricted to $k$.  We can obtain the following results: 
\begin{equation}
\label{e:wvwalksense}
\begin{split}
& w_{i,\critnodes}^{\ast}=p(\walkslen{i}{\critnodes}{l}{\trellis,\init})\quad \text{for any $l\geq \min\left(\frac{ w_{i,\critnodes}^{\ast} - \alpha_{i,\critnodes}}{\lambda_{q\ast}} + (n-q),k\right)$}\\
& v_{\critnodes,j}^{\ast}=p(\walkslen{\critnodes}{j}{m}{\trellis,\final})\quad \text{for any $m\geq \min\left(\frac{ v_{\critnodes,j}^{\ast} - \beta_{\critnodes,j}}{\lambda_{q\ast}} + (n-q),k\right)$}.
\end{split}
\end{equation}

\end{remark}

\if{
\begin{remark}
Observe that we do not have to fix $l$ and $m$ in Lemma~\ref{l:intrep} or $k$ in Lemma~\ref{l:intrepk}. In particular, we also have equalities
$$
w_{i,\critnodes}^{\ast}=p(\walkslen{i}{\critnodes}{l,\gamma}{\trellis}),\quad 
v_{\critnodes,j}^{\ast}=p(\walkslen{\critnodes}{j}{m,\gamma}{\trellis}).
$$
Here $\walkslen{i}{\critnodes}{l,\gamma}{\trellis}$ denotes the set of all initial walks that connect $i$ to $\critnodes$ on 
$\trellis$ and whose length is $l\modd\gamma$ and not exceeding $k$, and $\walkslen{\critnodes}{j}{m,\gamma}{\trellis}$ denotes the
set of all final walks that connect $\critnodes$ to $j$ on 
$\trellis$ and whose length is $m\modd\gamma$ and not exceeding $k$.
\end{remark}
}\fi 

\if{
\begin{remark}
It can be seen that the proof of the first part of~\eqref{eq:maxrep} uses that all critical arcs have weight $0$, and the proof of the second part of~\eqref{eq:maxrep} uses that the non-critical arcs have weight $\leq 0$. For the second part, we can do without this assumption, but only if the cyclicities of the critical graph and the ambient (associated) graph are the same (arguing that $V$ can be replaced with a critical walk, for which it should also be sufficiently long). 
\end{remark}
}\fi

\if{
There also exists a lemma for the optimal final walk. As the proof is analogous to the proof of Lemma~\ref{l:intrep} it will be omitted.
\begin{lemma} \label{l:finrep}
Let $\trellis(\Gamma(k))$ be the trellis digraph for a matrix product of length $k$ using matrices from $\mathcal{Y}$ with the associated digraph for each matrix element of $\mathcal{Y}$ being $\digr$. Let $\cclass_{j,k}$ be the set of nodes in $\critnodes$ such that there exists a walk on $\digr$ of length congruent to $k\modd{\gamma}$ connecting $m\in \cclass_{j,k}$ to $j$. If $W_{\text{opt}}$ is the optimal final walk connecting a node in $\critnodes$ to some node $j$ with weight $ v_{\critnodes,j}^{\ast}$, then
\begin{equation} \label{eq:finrep}
    v_{\critnodes,j}^{\ast} = \max_{m\in\cclass_{j,k}}\left( p\left( \walkslen{m}{j}{k}{\trellis}\right)\right).
\end{equation}
\end{lemma}
}\fi

\if{
\begin{lemma} \label{l:CSRwv}
Let $\Gamma(k)=A_{1} \otimes \ldots \otimes A_{k}$ be a matrix product of length $k$ using matrices from $\mathcal{Y}$ and let $CS^{k\modd{\gamma}} R[\Gamma(k)]$ be the CSR decomposition of $\Gamma(k)$ with respect to Definition~\ref{d:CSR1a} using $t\geq\gamma\text{Wi}\left(\left\lfloor \frac{q}{\gamma}\right\rfloor\right) + q\modd{\gamma}$. If $(CS^{k\modd{\gamma}} R[\Gamma(k)])_{i,j}\neq -\infty$ then $(CS^{k\modd{\gamma}} R[\Gamma(k)])_{i,j}=w_{i,\critnodes}^{\ast} + v_{\critnodes,j}^{\ast}$.
\end{lemma}
}\fi

We now establish the connection between the previous Lemma and CSR.
\begin{lemma}
\label{l:CSRwv}
We have one of the following cases:
\begin{itemize}
    \item[{\rm (i)}] $(CS^{k\modd{\gamma}}R[\Gamma(k)])_{i,j}=\tropzero$ if 
    $[i]'\not\to_k [j]'$,
    \item[{\rm (ii)}] $(CS^{k\modd{\gamma}}R[\Gamma(k)])_{i,j}=w_{i,\critnodes}^*+v_{\critnodes,j}^*$ if $[i]'\to_k[j]'$.
\end{itemize}
\end{lemma}

\begin{proof}
By Lemma~\ref{r:CSRrep} we have $p\left(\walkslen{i}{j}{2k+v}{\trellis',\full}\right)= (CS^{k\modd{\gamma}}R[\Gamma(k)])_{i,j}$, 
where $v=(t+1)\gamma-k\modd{\gamma}$ and $t\gamma\geq T(S)$,
and let $W\in \walkslen{i}{j}{2k+v}{\trellis',\full}$ be optimal. $W$ can be decomposed as $W_1W_2W_3$ where $W_1$ is a full walk (of length $k$) connecting $i$ to some $l\in\critnodes$ on $\trellis$, $W_{3}$ is a (full) walk of length $k$ connecting some $m\in \critnodes$ to $j$ and $W_{2}$ is a walk on the critical graph of length $v$ connecting the end of $W_{1}$ to the beginning of $W_{3}$. In formula,
\begin{equation}
\label{e:optwalkCSR2}
\begin{split}
   (CS^{k\modd{\gamma}}R[\Gamma(k)])_{i,j} &=
  \max\{p(W_1)+p(W_2)+p(W_3)\colon\\ & W_1\in\walkslen{i}{l}{k}{\trellis,\full},\  W_2\in\walkslen{l}{m}{v}{\crit},\ W_3\in\walkslen{m}{j}{k}{\trellis,\full},\ 
  l,m\in\critnodes\} 
\end{split}
\end{equation}
If the weights of $W_1$, $W_2$ and $W_3$ in~\eqref{e:optwalkCSR2} are finite then $[i]'\to_k [l]'$, $[l]'\to_v [m]'$ and $[m]'\to_k [j]'$, hence $[i]'\to_k [j]'$. Thus 
$(CS^tR[\Gamma(k)]_{i,j})> \tropzero$ implies $[i]'\to_k [j]'$ proving (i).

As the cyclicity of the associated graph is the same as the cyclicity of the critical graph, Lemma~\ref{l:intrep} implies that 
\begin{equation}
\label{e:astinterp} 
w_{i,\critnodes}^{\ast}=p(\walkslen{i}{\cclass_{i,k}}{k}{\trellis}),\quad  v_{\critnodes,j}^{\ast}=p(\walkslen{\cclass_{k,j}}{j}{k}{\trellis}),
\end{equation}
where $\cclass_{i,k}=\cclass'_{i,k}\cap\critnodes$ is the cyclic class of $\crit(\mathcal{X})$ that can be found by intersecting with critical nodes $\critnodes$ the cyclic class $\cclass'_{i,k}$ of $\digr$ defined by $[i]'\to_k \cclass'_{i,k}$. Similarly,  $\cclass_{k,j}=\cclass'_{k,j}\cap\critnodes$ is the cyclic class of $\crit(\mathcal{X})$ that can be found by intersecting with critical nodes $\critnodes$ the cyclic class $\cclass'_{k,j}$ of $\digr$ defined by $\cclass'_{k,j}\to_k [j]'$.  

Now note that in~\eqref{e:optwalkCSR2} we can similarly restrict $l$ to $\cclass_{i,k}$ and $m$ to $\cclass_{k,j}$,
which transforms it to
\begin{equation}
\label{e:optwalkCSR3}
\begin{split}
  (CS^{k\modd{\gamma}}R[\Gamma(k)])_{i,j} &=
  \max\{p(W_1)+p(W_2)+p(W_3)\colon\\ & W_1\in\walkslen{i}{l}{k}{\trellis},\  W_2\in\walkslen{l}{m}{v}{\crit},\ W_3\in\walkslen{m}{j}{k}{\trellis},\ 
  l\in\cclass_{i,k},\ m\in\cclass_{k,j}\} 
\end{split}
\end{equation}
Note that if a walk $W_2$ exists between any 
$l\in\cclass_{i,k}$ and $m\in\cclass_{k,j}$ then using~\eqref{e:astinterp} we immediately obtain 
$(CS^{k\modd{\gamma}}R[\Gamma(k)])_{i,j}=w^{\ast}_{i,\critnodes}+v^{\ast}_{\critnodes,j}$. Thus it remains to show existence of $W_2\in\walkslen{l}{m}{v}{\crit}$ between any $l\in\cclass_{i,k}$ and $m\in\cclass_{k,j}$.  
For this note that since $v= (t+1)\gamma-k\modd\gamma\geq T(S)$, either $\cclass_{i,k}\to_{(\gamma-k\modd\gamma)} \cclass_{k,j}$ and a walk on $\crit(\mathcal{X})$ of length $v$ exists between each pair of nodes in $\cclass_{i,k}$ and $\cclass_{k,j}$, or $\cclass_{i,k}\not\to_{(\gamma-k\modd\gamma)} \cclass_{k,j}$ and then no such walk exists. We thus have to check that $\cclass_{i,k}\to_{(\gamma-k\modd\gamma)} \cclass_{k,j}$ on $\digr$. But this follows since we have $[i]'\to_k [j]'$, and since in the sequence $[i]'\to_k \cclass'_{i,k}\to_l \cclass'_{k,j}\to_k [j]'$ we then must have $l\equiv_{\gamma} \gamma-k\modd\gamma$.

\if{
Therefore for any $l\in \tilde{C_{i}}$ and $m \in \tilde{C_{j}}$ we have
\begin{equation*}
    p(W_{l,m})=p\left(\walkslen{i}{l}{k}{\trellis}\right) + p\left(\walkslen{l}{m}{v}{\digr(S)}\right) + p\left(\walkslen{m}{j}{k}{\trellis}\right).
\end{equation*}
Let $P=\{W_{l,m}\vert l\in \tilde{C_{i}} \text{ and } m\in \tilde{C_{j}} \}$. It is easy to see that $W^{\ast}\in P$ and as $W^{\ast}$ is an optimal walk then its weight must be the largest weight of all the walks in $P$, hence
\begin{align*}
     p\left(\walkslennode{i}{j}{2k+v}{\trellis'}{\critnodes}\right)  & = \max_{l\in \tilde{C_{i}},m \in \tilde{C_{j}}} \left( p\left(\walkslen{i}{l}{k}{\trellis}\right) + p\left(\walkslen{l}{m}{v}{\digr(S)}\right) + p\left(\walkslen{m}{j}{k}{\trellis}\right)\right) \\
     & = \max_{l\in \tilde{C_{i}}} \left(p\left(\walkslen{i}{l}{k}{\trellis}\right)\right) + p\left(\walkslen{l}{m}{v}{\digr(S)}\right) + \max_{m \in \tilde{C_{j}}} \left(p\left(\walkslen{m}{j}{k}{\trellis}\right)\right)
\end{align*}
As the walk $W_{2}$ is situated entirely on $\digr(S)$ it must have weight $0$, so by Lemma~\ref{l:intrep} and~\ref{l:finrep},
\begin{equation*}
    \max_{l\in \tilde{C_{i}}} \left(p\left(\walkslen{i}{l}{k}{\trellis}\right)\right) + p\left(\walkslen{l}{m}{v}{\digr(S)}\right) + \max_{m \in \tilde{C_{j}}} \left(p\left(\walkslen{m}{j}{k}{\trellis}\right)\right) = w^{\ast}_{i,\critnodes} + 0 + v^{\ast}_{\critnodes,j}.
\end{equation*}
Thus the proof is complete.
}\fi
\end{proof}

Combining Theorem~\ref{th:ambientgamma} and Lemma~\ref{l:CSRwv} we
obtain the following result.

\begin{theorem} \label{th:ambientgammaCSR}
Denote $u_{i,\critnodes,j}^{\ast}= w_{i\critnodes}^{\ast}+v_{\critnodes,j}^{\ast}$. Let $k$ be greater than or equal to
\begin{align*} 
     \max\left(\max_{i,j} \frac{ u_{i,\critnodes,j}^{\ast} - \alpha_{i,\critnodes}  - \beta_{\critnodes,j} } + 2(n-q) + 
    \sch(\gamma,q),\, \max_{i,j\colon\gamma_{i,j}>\tropzero} 
    \frac{u_{i,\critnodes,j}^{\ast}-\gamma_{i,j}}{\lambda_{\ast}}+n-q+1\right)
\end{align*}
Then $\Gamma(k)= CS^{k\modd{\gamma}}R[\Gamma(k)]$. 
\end{theorem}

As with Theorem~\ref{th:weakcsr} this bound requires $\Gamma(k)$ in order to calculate the bound, which makes it implicit, but as with Corollary~\ref{c:weakcsrinf}  we can use $w_{i,\critnodes}\leq w^*_{i,\critnodes}$ and $v_{\critnodes,j}\leq v^*_{\critnodes,j}$ to give us an explicit bound.

\begin{corollary} \label{c:ambientgamma}
Denote $u_{i,\critnodes,j}= w_{i\critnodes}+v_{\critnodes,j}$. Let $k$ be greater than or equal to
\begin{align*} 
\max\left(\max_{i,j} \frac{ u_{i,\critnodes,j} - \alpha_{i,\critnodes}  - \beta_{\critnodes,j} }{\lambda_{\ast}} + 2(n-q)+
\sch(\gamma,q),\,
 \max_{i,j\colon\gamma_{i,j}>\tropzero}     \frac{u_{i,\critnodes,j}-\gamma_{i,j}}{\lambda_{\ast}}+n-q+1\right)
\end{align*}
Then $\Gamma(k)= CS^{k\modd{\gamma}}R[\Gamma(k)]$. 
\end{corollary}

We will now present an example of this bound in action.

Let $\digr(G)$ be the eight node digraph with the following structure:

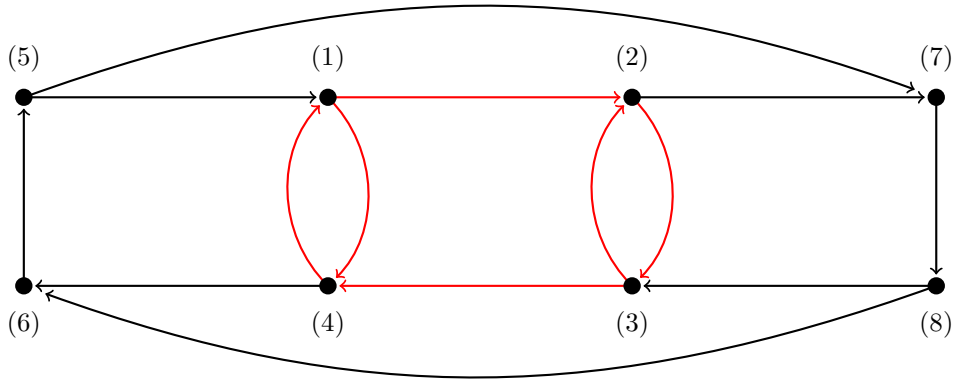
\begin{center}
 \begin{tikzpicture}[thick]
\coordinate (1) at (-2,1.25);
\coordinate (2) at (2,1.25);
\coordinate (3) at (2,-1.25);
\coordinate (4) at (-2,-1.25);
\coordinate (5) at (-6,1.25);
\coordinate (6) at (-6,-1.25);
\coordinate (7) at (6,1.25);
\coordinate (8) at (6,-1.25);

\draw [red, ->,shorten >= 0.15cm]   (1) to (2);
\draw [red, ->,shorten >= 0.15cm]   (2) to[out=-45,in=45] (3);
\draw [red, ->,shorten >= 0.15cm]   (3) to[out=135,in=-135] (2);
\draw [red, ->,shorten >= 0.15cm]   (3) to (4);
\draw [red, ->,shorten >= 0.15cm]   (4) to[out=135,in=-135] (1);
\draw [red, ->,shorten >= 0.15cm]   (1) to[out=-45,in=45] (4);

\draw [black, ->,shorten >= 0.15cm]   (4) to (6);
\draw [black, ->,shorten >= 0.15cm]   (6) to (5);
\draw [black, ->,shorten >= 0.15cm]   (5) to (1);
\draw [black, ->,shorten >= 0.15cm]   (2) to (7);
\draw [black, ->,shorten >= 0.15cm]   (7) to (8);
\draw [black, ->,shorten >= 0.15cm]   (8) to (3);
\draw [black, ->,shorten >= 0.3cm]   (5) to[out=20,in=160] (7);
\draw [black, ->,shorten >= 0.3cm]   (8) to[out=-160,in=-20] (6);

\draw[fill=black] (1) circle(0.1);
\node[label={$(1)$}] at (1) {a};
\draw[fill=black] (2) circle(0.1);
\node[label={$(2)$}] at (2) {a};
\draw[fill=black] (3) circle(0.1);
\node[label=below:{$(3)$}] at (3) {a};
\draw[fill=black] (4) circle(0.1);
\node[label=below:{$(4)$}] at (4) {a};
\draw[fill=black] (5) circle(0.1);
\node[label={$(5)$}] at (5) {a};
\draw[fill=black] (6) circle(0.1);
\node[label=below:{$(6)$}] at (6) {a};
\draw[fill=black] (7) circle(0.1);
\node[label={$(7)$}] at (7) {a};
\draw[fill=black] (8) circle(0.1);
\node[label=below:{$(8)$}] at (8) {a};
\end{tikzpicture}    
\end{center}
along with the associated weight matrix.
\begin{align*}
    A = \begin{pmatrix}
    \tropzero  &   0 & \tropzero &    0 & \tropzero & \tropzero & \tropzero & \tropzero \\
  \tropzero & \tropzero  &   0 & \tropzero &  \tropzero & \tropzero & A_{2,7} & \tropzero \\
  \tropzero &    0 & \tropzero &    0 & \tropzero & \tropzero & \tropzero & \tropzero \\
     0 & \tropzero & \tropzero & \tropzero & \tropzero &   A_{4,6} & \tropzero & \tropzero \\
     A_{5,1} & \tropzero & \tropzero & \tropzero & \tropzero & \tropzero &   A_{5,7} & \tropzero \\
  \tropzero & \tropzero & \tropzero & \tropzero &   A_{6,5} & \tropzero & \tropzero & \tropzero \\
  \tropzero & \tropzero & \tropzero & \tropzero & \tropzero & \tropzero & \tropzero &   A_{7,8} \\
  \tropzero & \tropzero &    A_{8,3} & \tropzero & \tropzero &    A_{8,6} & \tropzero & \tropzero 
    \end{pmatrix}
\end{align*}
There are three critical cycles in this digraph, one cycle of length $4$ traversing $1\to2\to3\to4$, and two cycles of length $2$ traversing $1\to4\to1$ and $2\to3\to2$ respectively. There are also cycles of length $4$, $6$ and $8$ which means that the cyclicity of the whole digraph is $2$, which is the same cyclicity of the critical subgraph. Therefore Assumption~\ref{as:ambientgamma} is satisfied and we can continue.

The semigroup of matrices $\mathcal{X}$ used by this example will be generated by these five matrices:
\small\begin{align*}
        A_{1} & = \begin{pmatrix}
    \tropzero  &   0 & \tropzero &    0 & \tropzero & \tropzero & \tropzero & \tropzero \\
  \tropzero & \tropzero  &   0 & \tropzero &  \tropzero & \tropzero & -16 & \tropzero \\
  \tropzero &    0 & \tropzero &    0 & \tropzero & \tropzero & \tropzero & \tropzero \\
     0 & \tropzero & \tropzero & \tropzero & \tropzero &   -6 & \tropzero & \tropzero \\
     -11 & \tropzero & \tropzero & \tropzero & \tropzero & \tropzero &  -14 & \tropzero \\
  \tropzero & \tropzero & \tropzero & \tropzero &   -18 & \tropzero & \tropzero & \tropzero \\
  \tropzero & \tropzero & \tropzero & \tropzero & \tropzero & \tropzero & \tropzero &   -20 \\
  \tropzero & \tropzero &    -11 & \tropzero & \tropzero &    -3 & \tropzero & \tropzero 
    \end{pmatrix},
        A_{2} = \begin{pmatrix}
    \tropzero  &   0 & \tropzero &    0 & \tropzero & \tropzero & \tropzero & \tropzero \\
  \tropzero & \tropzero  &   0 & \tropzero &  \tropzero & \tropzero & -3 & \tropzero \\
  \tropzero &    0 & \tropzero &    0 & \tropzero & \tropzero & \tropzero & \tropzero \\
     0 & \tropzero & \tropzero & \tropzero & \tropzero &   -6 & \tropzero & \tropzero \\
     -17 & \tropzero & \tropzero & \tropzero & \tropzero & \tropzero &   -6 & \tropzero \\
  \tropzero & \tropzero & \tropzero & \tropzero &  -17 & \tropzero & \tropzero & \tropzero \\
  \tropzero & \tropzero & \tropzero & \tropzero & \tropzero & \tropzero & \tropzero &   -5 \\
  \tropzero & \tropzero &   -19 & \tropzero & \tropzero &    -7 & \tropzero & \tropzero 
    \end{pmatrix}, \\
        A_{3} & = \begin{pmatrix}
    \tropzero  &   0 & \tropzero &    0 & \tropzero & \tropzero & \tropzero & \tropzero \\
  \tropzero & \tropzero  &   0 & \tropzero &  \tropzero & \tropzero & -4 & \tropzero \\
  \tropzero &    0 & \tropzero &    0 & \tropzero & \tropzero & \tropzero & \tropzero \\
     0 & \tropzero & \tropzero & \tropzero & \tropzero &   -6 & \tropzero & \tropzero \\
     -13 & \tropzero & \tropzero & \tropzero & \tropzero & \tropzero &   -10 & \tropzero \\
  \tropzero & \tropzero & \tropzero & \tropzero &   -8 & \tropzero & \tropzero & \tropzero \\
  \tropzero & \tropzero & \tropzero & \tropzero & \tropzero & \tropzero & \tropzero &   -17 \\
  \tropzero & \tropzero &    -12 & \tropzero & \tropzero &    -11 & \tropzero & \tropzero 
    \end{pmatrix},
        A_{4} = \begin{pmatrix}
    \tropzero  &   0 & \tropzero &    0 & \tropzero & \tropzero & \tropzero & \tropzero \\
  \tropzero & \tropzero  &   0 & \tropzero &  \tropzero & \tropzero & -19 & \tropzero \\
  \tropzero &    0 & \tropzero &    0 & \tropzero & \tropzero & \tropzero & \tropzero \\
     0 & \tropzero & \tropzero & \tropzero & \tropzero &   -6 & \tropzero & \tropzero \\
     -16 & \tropzero & \tropzero & \tropzero & \tropzero & \tropzero &   -16 & \tropzero \\
  \tropzero & \tropzero & \tropzero & \tropzero &   -8 & \tropzero & \tropzero & \tropzero \\
  \tropzero & \tropzero & \tropzero & \tropzero & \tropzero & \tropzero & \tropzero &   -12 \\
  \tropzero & \tropzero &    -2 & \tropzero & \tropzero &    -2 & \tropzero & \tropzero 
    \end{pmatrix}, \\
        A_{5} &  = \begin{pmatrix}
    \tropzero  &   0 & \tropzero &    0 & \tropzero & \tropzero & \tropzero & \tropzero \\
  \tropzero & \tropzero  &   0 & \tropzero &  \tropzero & \tropzero & -11 & \tropzero \\
  \tropzero &    0 & \tropzero &    0 & \tropzero & \tropzero & \tropzero & \tropzero \\
     0 & \tropzero & \tropzero & \tropzero & \tropzero &   -16 & \tropzero & \tropzero \\
     -19 & \tropzero & \tropzero & \tropzero & \tropzero & \tropzero &   -3 & \tropzero \\
  \tropzero & \tropzero & \tropzero & \tropzero &  -12 & \tropzero & \tropzero & \tropzero \\
  \tropzero & \tropzero & \tropzero & \tropzero & \tropzero & \tropzero & \tropzero &   -10 \\
  \tropzero & \tropzero &    -1 & \tropzero & \tropzero &    -7 & \tropzero & \tropzero 
    \end{pmatrix}.
\end{align*} \normalsize

Using these matrices we can calculate $\Asup$ and $\Ainf$, 
\small \begin{align*}
    \Asup & = \begin{pmatrix}
    \tropzero  &   0 & \tropzero &    0 & \tropzero & \tropzero & \tropzero & \tropzero \\
  \tropzero & \tropzero  &   0 & \tropzero &  \tropzero & \tropzero & -3 & \tropzero \\
  \tropzero &    0 & \tropzero &    0 & \tropzero & \tropzero & \tropzero & \tropzero \\
     0 & \tropzero & \tropzero & \tropzero & \tropzero &   -6 & \tropzero & \tropzero \\
     -11 & \tropzero & \tropzero & \tropzero & \tropzero & \tropzero &   -3 & \tropzero \\
  \tropzero & \tropzero & \tropzero & \tropzero &  -8 & \tropzero & \tropzero & \tropzero \\
  \tropzero & \tropzero & \tropzero & \tropzero & \tropzero & \tropzero & \tropzero &   -5 \\
  \tropzero & \tropzero &    -1 & \tropzero & \tropzero &    -2 & \tropzero & \tropzero 
    \end{pmatrix},
    \Ainf = \begin{pmatrix}
    \tropzero  &   0 & \tropzero &    0 & \tropzero & \tropzero & \tropzero & \tropzero \\
  \tropzero & \tropzero  &   0 & \tropzero &  \tropzero & \tropzero & -19 & \tropzero \\
  \tropzero &    0 & \tropzero &    0 & \tropzero & \tropzero & \tropzero & \tropzero \\
     0 & \tropzero & \tropzero & \tropzero & \tropzero &   -16 & \tropzero & \tropzero \\
     -19 & \tropzero & \tropzero & \tropzero & \tropzero & \tropzero &   -16 & \tropzero \\
  \tropzero & \tropzero & \tropzero & \tropzero &  -18 & \tropzero & \tropzero & \tropzero \\
  \tropzero & \tropzero & \tropzero & \tropzero & \tropzero & \tropzero & \tropzero &   -20 \\
  \tropzero & \tropzero &    -19 & \tropzero & \tropzero &    -11 & \tropzero & \tropzero 
    \end{pmatrix}
\end{align*} \normalsize
as well as $\alpha_{i,\critnodes}$, $\beta_{\critnodes,j}$, $\gamma_{i,j}$, $w_{i,\critnodes}$ and $v_{\critnodes,j}$:
\begin{align*}
    \alpha_{i,\critnodes} = \begin{pmatrix} 0 \\
0 \\
0 \\
0 \\
-9 \\
-17 \\
-6 \\
-1 
\end{pmatrix}, \quad \quad \quad  \beta_{\critnodes,j}^{T} = \begin{pmatrix} 0 \\
0 \\
0 \\
0 \\
-14 \\
-6 \\
-3 \\
-8 
\end{pmatrix}, \quad \quad \quad  \gamma_{i,j} = \begin{pmatrix}
    \tropzero & \tropzero & \tropzero & \tropzero & \tropzero & \tropzero & \tropzero & \tropzero \\
    \tropzero & \tropzero & \tropzero & \tropzero & \tropzero & \tropzero & \tropzero & \tropzero \\
    \tropzero & \tropzero & \tropzero & \tropzero & \tropzero & \tropzero & \tropzero & \tropzero \\
    \tropzero & \tropzero & \tropzero & \tropzero & \tropzero & \tropzero & \tropzero & \tropzero \\
    \tropzero & \tropzero & \tropzero & \tropzero &   -18  	& -10 &	-3 &	-8 \\
    \tropzero & \tropzero & \tropzero & \tropzero & -18	& -10 &	-3 &	-8 \\
    \tropzero & \tropzero & \tropzero & \tropzero & -15 &	-7 &	-18 &	-5 \\
    \tropzero & \tropzero & \tropzero & \tropzero & -10 &	-2 &	-13 &	-18 
    \end{pmatrix} \\
    \if{w_{i,\critnodes} = \begin{pmatrix} 0 \\
0 \\
0 \\
0 \\
-19 \\
-37 \\
-39 \\
-19 
\end{pmatrix},  v_{\critnodes,j}^{T} = \begin{pmatrix} 0 \\
0 \\
0 \\
0 \\
-34 \\
-16 \\
-19 \\
-39 
\end{pmatrix} \\}\fi
w_{i,\critnodes}^{T} = \begin{pmatrix} 0 &0 &0 &0 &-19 &-37 &-39 &-19 
\end{pmatrix},  v_{\critnodes,j} = \begin{pmatrix} 0 &0 &0 &0 &-34 &-16 &-19 &-39 
\end{pmatrix}.
\end{align*}
With all the pieces ready we can now form the bound of Corollary~\ref{c:ambientgamma},
\footnotesize\begin{align*}
\begin{split}
    k \geq &
    \max\left(\begin{pmatrix}
    12 &    12 &	12 &	12 &	16.4 &	14.2 &	15.6 &	18.9 \\
    12 & 	12 &	12 &	12 &	16.4 &	14.2 &	15.6 &	18.9 \\
    12 &	12 &	12 &	12 &	16.4 &	14.2 &	15.6 &	18.9 \\
    12 &	12 &	12 &	12 &	16.4 &	14.2 &	15.6 &	18.9 \\
    14.2 &	14.2 &	14.2 &	14.2 &	18.7 &	16.4 &	17.8 &	21.1 \\
    16.4 &	16.4 &	16.4 &	16.4 &	20.9 &	18.7 &	20 &	23.3 \\
    19.3 &	19.3 &	19.3 &	19.3 &	23.8 &	21.6 &	22.9 &	26.2 \\
    16 &	16 &	16   & 	16   &  20.4 &	18.22 &	19.6 &	22.9
    \end{pmatrix},
    \begin{pmatrix}
    \tropzero &	\tropzero &	\tropzero &	\tropzero &	\tropzero &	\tropzero &	\tropzero &	\tropzero \\
    \tropzero &	\tropzero &	\tropzero &	\tropzero &	\tropzero &	\tropzero &	\tropzero &	\tropzero \\
    \tropzero &	\tropzero &	\tropzero &	\tropzero &	\tropzero &	\tropzero &	\tropzero &	\tropzero \\
    \tropzero &	\tropzero &	\tropzero &	\tropzero &	\tropzero &	\tropzero &	\tropzero &	\tropzero \\
    \tropzero &	\tropzero &	\tropzero &	\tropzero &	12.8 &	10.6 &	12.8 &	16.1 \\
    \tropzero &	\tropzero &	\tropzero &	\tropzero &	19 &	12.8 &	15 &	18.3 \\
    \tropzero &	\tropzero &	\tropzero &	\tropzero &	17.9 &	15.7 &	13.9 &	21.2 \\
    \tropzero &	\tropzero &	\tropzero &	\tropzero &	14.6 &	12.3 &	10.6 &	13.9
    \end{pmatrix}\right) \\ \normalsize
    & \Rightarrow k \geq 23.8.
\end{split}
\end{align*}
Therefore by Corollary~\ref{c:ambientgamma} if the length of a product using the matrices from $\mathcal{X}$ is greater than or equal to $24$ then the resulting product will be CSR. We will show such a product. Let $\Gamma(24)$ be the inhomogeneous matrix product made using the word $P=5     5     1     5     4     1     2     3     5     5     1     5     5     3     5     1     3     5     4     5     4     1     5     5$ which gives us:
\begin{align*}
    \Gamma(24)=\begin{pmatrix}
    0 & \tropzero  &   0 & \tropzero & \tropzero &  -16 &  -11 & \tropzero \\
  \tropzero &    0 & \tropzero &     0 &  -28 & \tropzero &  \tropzero &  -21 \\
     0 & \tropzero &    0 & \tropzero & \tropzero &  -16 &  -11 & \tropzero \\
  \tropzero &    0 & \tropzero &    0 &  -28 & \tropzero & \tropzero &  -21 \\
  \tropzero &  -19 & \tropzero &  -19 &  -47 & \tropzero & \tropzero &  -40 \\
   -31 & \tropzero &  -31 & \tropzero & \tropzero &  -47 &  -42 & \tropzero \\
   -11 & \tropzero &  -11 & \tropzero & \tropzero &  -27 &  -22 & \tropzero \\
  \tropzero &   -1 & \tropzero &   -1 &  -29 & \tropzero & \tropzero &  -22
    \end{pmatrix}.
\end{align*}
This matrix product is indeed CSR and by Definition~\ref{d:CSR1} we have,
\begin{align*}
    \Gamma(24) & = \begin{pmatrix}
     0 & \tropzero &    0 & \tropzero \\
  \tropzero &    0 & \tropzero &    0 \\
     0 & \tropzero &    0 & \tropzero \\
  \tropzero &    0 & \tropzero &    0 \\
  \tropzero &  -19 & \tropzero &  -19 \\
   -31 & \tropzero &  -31 & \tropzero \\
   -11 & \tropzero &  -11 & \tropzero \\
  \tropzero &   -1 & \tropzero &   -1 
    \end{pmatrix} \otimes
    \begin{pmatrix}
    0 & \tropzero & \tropzero & \tropzero \\
    \tropzero & 0 & \tropzero & \tropzero \\
    \tropzero & \tropzero & 0 & \tropzero \\
    \tropzero & \tropzero & \tropzero & 0 
    \end{pmatrix} \otimes
    \begin{pmatrix}
     0 & \tropzero &    0 & \tropzero & \tropzero &  -16  & -11 & \tropzero \\
  \tropzero &    0 & \tropzero &    0 &  -28 & \tropzero  & \tropzero &  -21 \\
     0 & \tropzero &    0 & \tropzero & \tropzero &  -16  & -11 & \tropzero \\
  \tropzero &    0 & \tropzero &    0 &  -28 &  \tropzero & \tropzero &  -21
    \end{pmatrix} \\
    \Gamma(24) & = \begin{pmatrix}
     0 & \tropzero \\
  \tropzero &    0  \\
     0 & \tropzero  \\
  \tropzero &    0 \\
  \tropzero &  -19  \\
   -31 & \tropzero  \\
   -11 & \tropzero  \\
  \tropzero &   -1  
    \end{pmatrix} \otimes
    \begin{pmatrix}
    0 & \tropzero \\
    \tropzero & 0
    \end{pmatrix} \otimes
    \begin{pmatrix}
     0 & \tropzero &    0 & \tropzero & \tropzero &  -16  & -11 & \tropzero \\
  \tropzero &    0 & \tropzero &    0 &  -28 & \tropzero  & \tropzero &  -21 
    \end{pmatrix}.
\end{align*}
We can see that, for the $C$ matrix, columns $3$ and $4$ are copies of columns $1$ and $2$ respectively. The same is also true for the rows of the $R$ matrix so they can be deleted. As $24\modd{2}=0$ we replace the $S$ matrix with the tropical identity matrix which shows us that the matrix product $\Gamma(24)$ using the word $P$ is indeed CSR and it has factor rank-$2$.

\section{Counterexamples} \label{cpt:nobound}
Here we present a number of counterexamples for the different cases of digraph structure. These counterexamples present families of products which are not CSR, and we construct them in such a way that they have no upper bound on their length.

\subsection{The ambient graph is primitive but the 
critical graph is not} 

We will now look at two cases where we are unable to create a bound for matrix products to become CSR. For the first case we will be looking at digraphs that are primitive but have a critical subgraph with a non-trivial cylicity. Therefore we have the following assumption:

\begin{asssumption}{P}{1} \label{a:ambientprimative}
$\digr(\mathcal{X})$ is primitive (i.e., $\gamma(\digr(\mathcal{X}))=1$) and the critical subgraph $\crit(\mathcal{X})$, which is a single strongly connected component, has cyclicity $\gamma(\crit(\mathcal{X}))=\gamma$.
\end{asssumption}

Using these assumptions we can now present a counterexample which shows that no bound for $k$ in terms of $\Asup$ and $\Ainf$ can exist that ensures that $\Gamma(k)$ is CSR.

Let $\digr(G)$ be the five node digraph with the following structure:
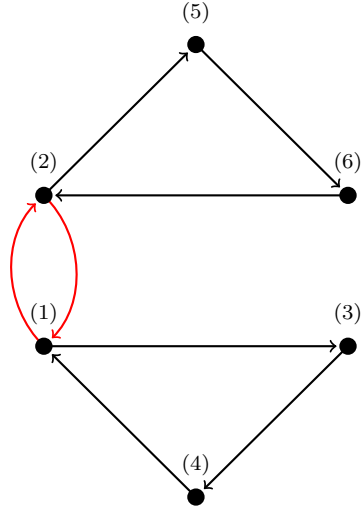
\begin{center}
 \begin{tikzpicture}[thick]
\coordinate (5) at (2,3);
\coordinate (6) at (4,1);
\coordinate (4) at (2,-3);
\coordinate (2) at (0,1);
\coordinate (1) at (0,-1);
\coordinate (3) at (4,-1);

\draw [black, ->,shorten >= 0.15cm]   (2) to (5);
\draw [black, ->,shorten >= 0.15cm]   (5) to (6);
\draw [black, ->,shorten >= 0.15cm]   (6) to (2);
\draw [black, ->,shorten >= 0.15cm]   (1) to (3);
\draw [black, ->,shorten >= 0.15cm]   (3) to (4);
\draw [black, ->,shorten >= 0.15cm]   (4) to (1);
\draw [red, ->,shorten >= 0.15cm]   (2) to[out=-45,in=45] (1);
\draw [red, ->,shorten >= 0.15cm]   (1) to[out=135,in=-135] (2);

\draw[fill=black] (1) circle(0.1);
\node[label={$(1)$}] at (1) {a};
\draw[fill=black] (2) circle(0.1);
\node[label={$(2)$}] at (2) {a};
\draw[fill=black] (3) circle(0.1);
\node[label={$(3)$}] at (3) {a};
\draw[fill=black] (4) circle(0.1);
\node[label={$(4)$}] at (4) {a};
\draw[fill=black] (5) circle(0.1);
\node[label={$(5)$}] at (5) {a};
\draw[fill=black] (6) circle(0.1);
\node[label={$(6)$}] at (6) {a};
\end{tikzpicture}    
\end{center}

This digraph will have the following associated weight matrix.
\begin{align*}
A= \begin{pmatrix}
\tropzero &   0         &    A_{1,3}   & \tropzero & \tropzero & \tropzero \\
   0        & \tropzero & \tropzero & \tropzero &    A_{2,5}   & \tropzero \\
\tropzero & \tropzero & \tropzero &    A_{3,4}   & \tropzero &    A_{3,6}   \\
   A_{4,1}   & \tropzero & \tropzero & \tropzero & \tropzero & \tropzero \\
\tropzero & \tropzero & \tropzero & \tropzero & \tropzero &    A_{5,6}   \\
\tropzero &  A_{6,2}     &    A_{6,3}   & \tropzero & \tropzero & \tropzero 
\end{pmatrix}
\end{align*}
There is a critical subgraph consisting of the cycle between nodes $1$ and $2$. There also exist two cycles, $1 \to 3 \to 4 \to 1$ and $2 \to 5 \to 6 \to 2$, both of length $3$ which makes $\digr(A)$  primitive. We aim to present a family of words with infinite length such that the products made up using these words are not CSR. Since the cyclicity of the critical subgraph is $2$ then we will have to create two classes of words, one of even length and one of odd length to define the family.

The semigroup of matrices we will use is generated by the two matrices:
\small\begin{align*}
A_{1}= \begin{pmatrix}
        \tropzero &          0     &  -100 &       \tropzero     &   \tropzero &       \tropzero \\
           0  &      \tropzero     &   \tropzero   &     \tropzero     &  -100   &     \tropzero \\
        \tropzero   &     \tropzero    &    \tropzero    &   -100    &    \tropzero    &   \tropzero \\
       -100    &    \tropzero   &     \tropzero     &   \tropzero   &     \tropzero     &   \tropzero \\
        \tropzero     &   \tropzero  &      \tropzero      &  \tropzero  &      \tropzero      & -100 \\
        \tropzero      & -100 &      \tropzero       & \tropzero &       \tropzero       & \tropzero \\
\end{pmatrix},
A_{2} = \begin{pmatrix}
        \tropzero &          0      & -100 &       \tropzero      &  \tropzero &       \tropzero \\
           0  &      \tropzero     &   \tropzero  &      \tropzero     &  -1  &      \tropzero \\
        \tropzero   &     \tropzero    &    \tropzero   &    -100    &    \tropzero   &    \tropzero \\
       -1    &    \tropzero   &     \tropzero    &    \tropzero   &     \tropzero    &    \tropzero \\
        \tropzero     &   \tropzero  &      \tropzero     &   \tropzero  &      \tropzero     &  -100 \\
        \tropzero      & -100 &      \tropzero      &  \tropzero &       \tropzero      &  \tropzero \\
\end{pmatrix}
\end{align*} \normalsize

Let us first consider the class of words 
$(1)^{2t}2$ where $t\geq 2$, and let $U=(A_1)^{2t}A_2$ for 
arbitrary such $t$. We will first examine entries $U_{6,1}$, $U_{2,5}$, $U_{6,2}$ and $U_{1,5}$.

The entry $U_{6,1}$ can be obtained as the weight of the walk $6\underbrace{(21)(21)\ldots (21)}_{t-1}341$, which is $-301$. For this observe that the walk $621$ has an even length and therefore we need to use one of the three-cycles to make it odd, and using the southern three-cycle in the end of the walk is the most profitable way to do so.  The entry $U_{25}$ is equal to $-1$, as there is a walk that mostly rests on the critical cycle and only in the end jumps to node $5$. We also have $U_{6,2}=-100$ (go to node $2$ and remain on the critical cycle) and $U_{1,5}=-301$ (use the southern triangle once, then dwell on the critical cycle and in the end jump to node $5$). Note that in the case of $U_{1,5}$ we again need to use one of the triangles to create a walk of an odd length.

We then compute
$$(CSR)[U]_{6,5}= (US^3U)_{6,5}=\max(U_{6,1}+U_{2,5},\; U_{6,2}+U_{1,5})
= -301-1=-302.$$

However, $U_{6,5}$ results from the walk
$6\underbrace{(21)(21)\ldots (21)}_{t-1}2562$, 
with weight $-401$, needing to use the northern triangle to make a walk of odd length. 

The following an example of $U$ and $CS^{2t+1}R[U]$ for $t=10$:

\begin{align*}
U=& \begin{pmatrix}
-201 &  0 & -100 & -500 & -301 & -200 \\
     0 &  -300 & -400 & -200  &  -1 & -500 \\
  -401 & -200 & -300 & -700 & -501 & -400 \\
  -100 & -400 & -500 & -300 & -101 & -600 \\
  -200 & -500 & -600 & -400 & -201 & -700 \\
  -301 & -100 & -200 & -600 & -401 & -300 
\end{pmatrix} \\
CS^{21\modd{2}}R[U]=& \begin{pmatrix}
-201   &  0 & -100 & -401 & -202 & -200 \\
     0 & -300 & -400 & -200 &   -1 & -500 \\
  -401 & -200 & -300 & -601 & -402 & -400 \\
  -100 & -400 & -500 & -300 & -101 & -600 \\
  -200 & -500 & -600 & -400 & -201 & -700 \\
  -301 & -100 & -200 & -501 & -302 & -300
\end{pmatrix}    
\end{align*}

We now consider the class of words $(1)^{2t+1}2$ where $t\geq 1$, and let 
$V=(A_1)^{2t+1}A_2$ for arbitrary such $t$.
We will first examine entries $V_{2,1}$, $V_{1,5}$, $V_{2,2}$ and $V_{2,5}$. 

The entry $V_{2,1}=-201$ is obtained as the weight of the walk $2\underbrace{(12)(12)\ldots (12)}_{t-1}341$: it is necessary to use one of the triangles to create a walk of even length, and using the southern triangle once in the end of the walk is the most profitable way to do so. The walk $125$ already has an even length, and we only have to augment it with enough copies of the critical cycle and use the arc $2\to 5$ in the end of the walk, thus getting $V_{1,5}=-1$. Obviously, $V_{2,2}=0:$ we just stay on the critical cycle. The entry $V_{2,5}=-301$ is obtained as the weight of the walk $\underbrace{(21)(21)\ldots (21)}_{t-1}5625$, where we have to use the northern triangle in the end of the walk to create a walk of even walk and minimise the loss.  

We then find 
$$
(CS^2R[V])_{2,5}=(VS^2V)_{2,5}=\max(V_{2,1}+V_{1,5}, V_{2,2}+V_{2,5})=V_{2,1}+V_{1,5}=-202, 
$$
which is bigger than $V_{2,5}=-301$.

The case for $V_{2,5}$ is one for connecting a critical node to a non critical node. For completeness we should also look at a walk connecting two non critical nodes, namely the walk representing $V_{4,5}$. To do this we will need to also look at the entries $V_{4,1}$ and $V_{4,2}$. For $V_{4,1}=-301$ the entry is obtained as the weight of the walk $4\underbrace{(12)(12)\ldots(12)}_{t-1}341$. As the walk $41$ has odd length, one of the triangles is required to make the walk even so choosing the southern triangle is the most profitable way to achieve an even length walk. The walk $412$ already has an even length so we can augment it with enough copies of the critical cycle to give us the desired length for the walk representing the entry $V_{4,2}=-100$. Using $V_{1,5}$ and $V_{2,5}$ discussed earlier we calculate
$$
(CS^2R[V])_{4,5}=(VS^2V)_{4,5}=\max(V_{4,1}+V_{1,5}, V_{4,2}+V_{2,5})=V_{4,1}+V_{1,5}=-302, 
$$
which is bigger than $V_{4,5}=-401$.

We now show an example of $V$ for $t=10$:

\begin{align*}
V=& \begin{pmatrix}
     0 & -300 & -400 & -200  &  -1 & -500 \\
  -201 &    0 & -100 & -500 & -301  & -200 \\
  -200 & -500 & -600 & -400 & -201  & -700 \\
  -301 & -100 & -200 & -600  &-401 & -300 \\
  -401 & -200 & -300 & -700 & -501 & -400 \\
  -100 & -400 & -500 & -300 & -101 & -600 
\end{pmatrix} \\
CS^{22\modd{2}}R[V]= & \begin{pmatrix}
     0 & -300 & -400 & -200 &   -1 & -500 \\
  -201 &    0 & -100 & -401 & -202 & -200 \\
  -200 & -500 & -600 & -400 & -201 & -700 \\
  -301 & -100 & -200 & -501 & -302 & -300 \\
  -401 & -200 & -300 & -601 & -402 & -400 \\
  -100 & -400 & -500 & -300 & -101 & -600
\end{pmatrix}    
\end{align*}

Combining both classes we have a family of words covering all lengths greater than $29$ such that any product made using these words will not be CSR. Therefore there cannot be a transient for this case as there is no upper limit to the lengths of these words. 

There also exists another counterexample in the primitive case which shows that even walks connecting two nodes from the same critical subgraph can not be CSR.

Let $\digr(G)$ be the three node digraph with the following structure:
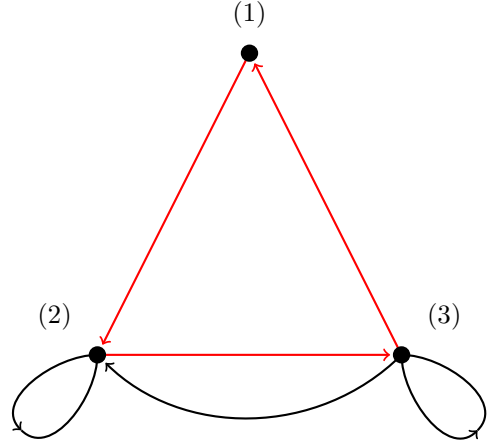
\begin{center}
 \begin{tikzpicture}[thick]
\coordinate (1) at (0,2);
\coordinate (2) at (-2,-2);
\coordinate (2a) at (-3,-3);
\coordinate (3) at (2,-2);
\coordinate (3a) at (3,-3);

\draw [red, ->,shorten >= 0.15cm]   (1) to (2);
\draw [red, ->,shorten >= 0.15cm]   (2) to (3);
\draw [red, ->,shorten >= 0.15cm]   (3) to (1);

\draw [black, ->]   (2) to[out=180,in=135] (2a);
\draw [black, ->]   (3) to[out=-90,in=-135] (3a);
\draw [black]   (2a) to[out=-45,in=-90] (2);
\draw [black]   (3a) to[out=45,in=0] (3);
\draw [black, ->,shorten >= 0.15cm]   (3) to[out=-135,in=-45] (2);

\draw[fill=black] (1) circle(0.1);
\node[label={$(1)$}] at (1) {a};
\draw[fill=black] (2) circle(0.1);
\node[label=above left:{$(2)$}] at (2) {a};
\draw[fill=black] (3) circle(0.1);
\node[label=above right:{$(3)$}] at (3) {a};

\end{tikzpicture}    
\end{center}
The digraph has the following associated weight matrix.
\begin{equation*}
    A = \begin{pmatrix}
    \tropzero & 0 & \tropzero \\
    \tropzero & A_{2,2} & 0 \\
    0 & A_{3,2} & A_{3,3} 
    \end{pmatrix}.
\end{equation*}
For this example there is a single critical cycle of length $3$ traversing all of the nodes. There also exists two loops $2\to2$ and $3\to3$ and a cycle $2\to3\to2$ of length $2$. Like the previous example this digraph is primitive but the critical subgraph has cyclicity $3$. As the cyclicity is greater than one we need to present three different classes of words making up a family of words such that any product $\Gamma(k)$ made using these words will not be CSR. 

The semigroup of matrices that we will use is again generated only by two matrices:
\begin{equation*}
    A_{1}=\begin{pmatrix}
    \tropzero & 0 & \tropzero \\
    \tropzero & -100 & 0 \\
    0 & -100 & -100 
    \end{pmatrix} \quad A_{2} = \begin{pmatrix}
    \tropzero & 0 & \tropzero \\
    \tropzero & -1 & 0 \\
    0 & -100 & -1 
    \end{pmatrix}
\end{equation*}

Let the first class of words be $(1)^{3t+2}2$ for $t\geq 0$, and let $M=(A_{1})^{3t+2}A_{2}$ for any arbitrary $t$. We will now examine the entries $M_{1,1}$, $M_{1,2}$, $M_{2,2}$ $M_{1,3}$ and $M_{3,2}$.

Since all the walks are of length $0$ modulo $3$ then any walk connecting $i$ to $i$ will have weight zero as we can simply use the critical cycle. This gives $M_{1,1}=M_{2,2}=0$. The entry $M_{1,2}$ can be obtained as the weight of the walk $(123)^{t+1}2$ which is $-100$. In this entry observe that the walk $12$ is of length $1$ modulo $3$ therefore we need to use the two cycle $2\to3\to2$ to give us a walk of the desired length. The entry $M_{1,3}$ is equal to the weight of the walk $(123)^{t+1}3$ and the entry $M_{3,2}$ is equal to the weight of the walk $(312)^{t+1}2$. For these entries observe that the walks $123$ and $312$ are both of length $2$ modulo $3$ therefore we require a loop for both walks to give us the required length. The most profitable time to use these loops are right at the end of the walk.

We then compute
$$(CSR)[M]_{1,2}= (MS^3M)_{1,2}=\max(M_{1,1} + M_{1,2}, M_{1,2} + M_{2,2}, M_{1,3} + M_{3,2})
= -1-1=-2.$$

However, as seen earlier the entry $M_{12}$ has weight $-100$ which is less than the CSR suggestion.

The following is an example of $M$ and $CS^{3t+3}R[M]$ for $t=10$:
\begin{equation*}
    M = \begin{pmatrix}
       0 & -100 &   -1 \\
  -100   &  0 & -100 \\
  -100   & -1   &  0 
    \end{pmatrix} \quad  CS^{33\modd{3}}R[M] = \begin{pmatrix}
         0  &  -2   & -1 \\
  -100  &   0 & -100 \\
  -100 &   -1  &   0
    \end{pmatrix}
\end{equation*}

For efficiency we will simply present the final two classes and omit the in-depth analysis of them:
\begin{itemize}
    \item[] For walks of length $1$ modulo $3$ we have the class of words $(1)^{3t+3}2$ for $t\geq 0$.
    \item[] For walks of length $2$ modulo $3$ we have the class of words $(1)^{3t+4}2$ for $t\geq 0$.
\end{itemize}
We will also present examples of products and their CSR counterparts made using these words for $t=10$ where $N=(A_{1})^{3t+3}A_{2}$ and $P=(A_{1})^{3t+4}A_{2}$.
\begin{align*}
    N = &  \begin{pmatrix}
  -100  &   0 & -100 \\
  -100  &  -1   &  0 \\
     0  & -100  &  -1
    \end{pmatrix}
    \quad
    CS^{34\modd{3}}R[N] = \begin{pmatrix}
  -100  &   0 & -100 \\
  -100  &  -1 &    0 \\
     0  &  -2 &   -1
    \end{pmatrix} \\
    P = &  \begin{pmatrix}
  -100  &  -1 &    0 \\
     0 & -100 &   -1 \\
  -100 &    0 & -100
    \end{pmatrix}
    \quad
    CS^{35\modd{3}}R[P] = \begin{pmatrix}
  -100  &  -1  &   0 \\
     0  &  -2  &  -1 \\
  -100  &   0 & -100
    \end{pmatrix}.
\end{align*}
The combination of these three classes create a family of words such that any product $\Gamma(k)$ made using these words is not CSR and as all the nodes are critical then there exist walks connecting them that are not CSR.

We now extend these counterexamples to a more general form where we consider digraphs with non-trivial cyclicity $r$ along with critical subgraphs with cyclicity $\gamma$ which is greater than $r$. This leads to the following assumptions. 

\subsection{More general case}


\begin{asssumption}{P}{2} \label{a:ambientless}
$\digr(\mathcal{X})$ has cyclicity $r$ and the critical subgraph $\crit(\mathcal{X})$, which strongly connected, has cyclicity $\gamma > r$.
\end{asssumption}

In a similar method to the primitive example above, using the new assumptions, we can now describe a counterexample that shows that no bound for $k$ in terms of $\Asup$ and $\Ainf$ can exist that ensures $\Gamma(k)$ is CSR.

Let $\digr(\mathcal{X})$ be a six node digraph with the following structure: 
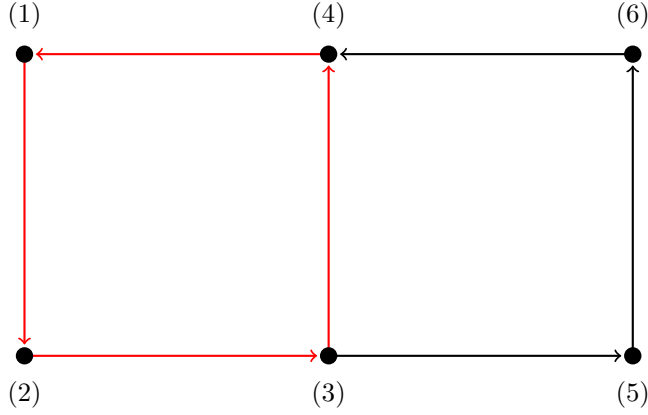
\begin{center}
 \begin{tikzpicture}[thick]
\coordinate (1) at (-2,2);
\coordinate (2) at (-2,-2);
\coordinate (3) at (2,-2);
\coordinate (4) at (2,2);
\coordinate (5) at (6,-2);
\coordinate (6) at (6,2);

\draw [red, ->,shorten >= 0.15cm]   (1) to (2);
\draw [red, ->,shorten >= 0.15cm]   (2) to (3);
\draw [red, ->,shorten >= 0.15cm]   (3) to (4);
\draw [red, ->,shorten >= 0.15cm]   (4) to (1);

\draw [black, ->,shorten >= 0.15cm]   (3) to (5);
\draw [black, ->,shorten >= 0.15cm]   (5) to (6);
\draw [black, ->,shorten >= 0.15cm]   (6) to (4);

\draw[fill=black] (1) circle(0.1);
\node[label={$(1)$}] at (1) {a};
\draw[fill=black] (2) circle(0.1);
\node[label=below:{$(2)$}] at (2) {a};
\draw[fill=black] (3) circle(0.1);
\node[label=below:{$(3)$}] at (3) {a};
\draw[fill=black] (4) circle(0.1);
\node[label={$(4)$}] at (4) {a};
\draw[fill=black] (5) circle(0.1);
\node[label=below:{$(5)$}] at (5) {a};
\draw[fill=black] (6) circle(0.1);
\node[label={$(6)$}] at (6) {a};

\end{tikzpicture}    
\end{center}
along with the following associated weight matrix,
\begin{align*}
A= \begin{pmatrix}
\tropzero &      0      & \tropzero & \tropzero & \tropzero & \tropzero  \\
\tropzero & \tropzero &      0      & \tropzero & \tropzero & \tropzero  \\
\tropzero & \tropzero & \tropzero &      0      &   A_{3,5}    & \tropzero  \\
     0      & \tropzero & \tropzero & \tropzero & \tropzero & \tropzero  \\
\tropzero & \tropzero & \tropzero & \tropzero & \tropzero &   A_{5,6}     \\
\tropzero & \tropzero & \tropzero &   A_{6,4}    & \tropzero & \tropzero  
\end{pmatrix}
\end{align*}
Here the critical cycle traverses nodes $1 \to 2 \to 3 \to 4 \to 1$ however there also exists another non-critical cycle of length six traversing $1 \to 2 \to 3 \to 5 \to 6 \to 4 \to 1$. This means that while the cyclicity of the critical subgraph is $4$ the cyclicity of $\digr(G)$ is $2$. Therefore the digraph structure satisfies the assumptions and we can develop a family of words with infinite length such that any $\Gamma(k)$ made using these words will not be CSR. As the cyclicity of the critical subgraph is $4$ then we will require four classes of words to fully define the family. 

The semigroup of matrices that will be used is generated by two matrices:
\begin{equation*}
A_1= \begin{pmatrix}
\tropzero &      0      & \tropzero & \tropzero & \tropzero & \tropzero  \\
\tropzero & \tropzero &      0      & \tropzero & \tropzero & \tropzero  \\
\tropzero & \tropzero & \tropzero &      0      &    -100     & \tropzero  \\
     0      & \tropzero & \tropzero & \tropzero & \tropzero & \tropzero  \\
\tropzero & \tropzero & \tropzero & \tropzero & \tropzero &    -100      \\
\tropzero & \tropzero & \tropzero &    -100     & \tropzero & \tropzero  
\end{pmatrix}
\quad
A_2= \begin{pmatrix}
\tropzero &      0      & \tropzero & \tropzero & \tropzero & \tropzero  \\
\tropzero & \tropzero &      0      & \tropzero & \tropzero & \tropzero  \\
\tropzero & \tropzero & \tropzero &      0      &     -1      & \tropzero  \\
     0      & \tropzero & \tropzero & \tropzero & \tropzero & \tropzero  \\
\tropzero & \tropzero & \tropzero & \tropzero & \tropzero &    -100      \\
\tropzero & \tropzero & \tropzero &     -1      & \tropzero & \tropzero  
\end{pmatrix}
\end{equation*}
Let us begin with the first class of words $(1)^{4t}2$ where $t\geq2$, and let $L=(A_{1})^{4t}A_{2}$ for arbitrary such $t$. We will begin by examining the entries $L_{1,2}$, $L_{1,5}$, $L_{1,4}$ and $L_{3,5}$.

The entry $L_{1,2}$ can be obtained as the weight of the walk $\underbrace{(1234)}_{t}12$, which is $0$. As the walk $12$ has length congruent to $1\modd{4}$ then a walk exists on the critical cycle connecting these nodes. The entry $L_{1,5}$ is obtained from the weight of the walk $\underbrace{(1234)}_{t-2}1235641235$, which is $-301$. As the walk $1235$ has length congruent to $3\modd{4}$ then we need to add on the six cycle with weight $-300$ to give us a walk of length congruent to $1\modd{4}$ and finally the last step of the walk is to go from $3$ to $5$ with weight $-1$. For the entry $L_{1,4}=-201$ which is the weight of the walk $\underbrace{(1234)}_{t-1}123564$ and the entry $L_{35}=-1$ comes from the weight of the walk $\underbrace{(3412)}_{t}35$. Note that in the case of $L_{1,4}$ we used the six cycle to give us the desired length of walk. 

We then compute
\begin{equation*}
    (CSR)[L]_{1,5} = (L\otimes S^3\otimes L)_{1,5} = \max(L_{1,2}+L_{1,5},\; L_{1,4} + L_{3,5}) = -201 -1 = -202.
\end{equation*}

However $L_{15}$, as explained earlier, results from a walk with weight $-301$.

The following is an example of $L$ and $CS^{4t+1}R[L]$ for $t=10$
\begin{align*}
L=& \begin{pmatrix}
  \tropzero &   0  & \tropzero & -201 & -301 & \tropzero \\
  -300 & \tropzero &    0 & \tropzero & \tropzero & -401 \\
  \tropzero & -300 & \tropzero &    0 &   -1 & \tropzero \\
     0 & \tropzero & -300 & \tropzero & \tropzero & -101 \\
  -500 & \tropzero & -200 & \tropzero & \tropzero & -601 \\
  \tropzero & -400 & \tropzero & -100 & -101 & \tropzero 
\end{pmatrix} \\
CS^{41\modd{4}}R[L]=& \begin{pmatrix}
  \tropzero &   0  & \tropzero & -201 & -202 & \tropzero \\
  -300 & \tropzero &    0 & \tropzero & \tropzero & -401 \\
  \tropzero & -300 & \tropzero &    0 &   -1 & \tropzero \\
     0 & \tropzero & -300 & \tropzero & \tropzero & -101 \\
  -500 & \tropzero & -200 & \tropzero & \tropzero & -601 \\
  \tropzero & -400 & \tropzero & -100 & -101 & \tropzero 
\end{pmatrix}
\end{align*}

The other classes behave in a similar way so we omit the in depth explanation of them. We present the words used for each class:
\begin{itemize}
    \item[] For walks of length congruent to $2\modd{4}$ we have the words $(1)^{4t+1}2$ for $t\geq 2$;
    \item[] For walks of length congruent to $3\modd{4}$ we have the words $(1)^{4t+2}2$ for $t\geq 2$;
    \item[] For walks of length congruent to $0\modd{4}$ we have the words $(1)^{4t+3}2$ for $t\geq 2$.
\end{itemize}

For example, if $t=10$ then for the first of these classes 
\begin{align*}
  & F=  (A_{1})^{41}\otimes A_{2}=  \begin{pmatrix}
  -300 & \tropzero &    0 & \tropzero & \tropzero & -401 \\
  \tropzero & -300 & \tropzero &    0 &   -1 & \tropzero \\
     0 & \tropzero & -300 & \tropzero & \tropzero & -101 \\
  \tropzero &    0 & \tropzero & -201 & -301 & \tropzero \\
  \tropzero & -500 & \tropzero & -200 & -201 & \tropzero \\
  -100 & \tropzero & -400 & \tropzero & \tropzero & -201 
\end{pmatrix},\\ 
& CS^{42\modd{4}}R[F]=  \begin{pmatrix}
  -300 & \tropzero &    0 & \tropzero & \tropzero & -401 \\
  \tropzero & -300 & \tropzero &    0 &   -1 & \tropzero \\
     0 & \tropzero & -300 & \tropzero & \tropzero & -101 \\
  \tropzero &    0 & \tropzero & -201 & -202 & \tropzero \\
  \tropzero & -500 & \tropzero & -200 & -201 & \tropzero \\
  -100 & \tropzero & -400 & \tropzero & \tropzero & -201  
\end{pmatrix} 
\end{align*}

\if{
We also present examples of each matrix product for $t=10$ and their CSR counterparts where $F=(A_{1})^{4t+1}A_{2}$, $G=(A_{1})^{4t+2}A_{2}$ and $H=(A_{1})^{4t+3}A_{2}$.
\begin{align*}
    F= & \begin{pmatrix}
  -300 & \tropzero &    0 & \tropzero & \tropzero & -401 \\
  \tropzero & -300 & \tropzero &    0 &   -1 & \tropzero \\
     0 & \tropzero & -300 & \tropzero & \tropzero & -101 \\
  \tropzero &    0 & \tropzero & -201 & -301 & \tropzero \\
  \tropzero & -500 & \tropzero & -200 & -201 & \tropzero \\
  -100 & \tropzero & -400 & \tropzero & \tropzero & -201 
\end{pmatrix},\\ 
CS^{42\modd{4}}R[F]= & \begin{pmatrix}
  -300 & \tropzero &    0 & \tropzero & \tropzero & -401 \\
  \tropzero & -300 & \tropzero &    0 &   -1 & \tropzero \\
     0 & \tropzero & -300 & \tropzero & \tropzero & -101 \\
  \tropzero &    0 & \tropzero & -201 & -202 & \tropzero \\
  \tropzero & -500 & \tropzero & -200 & -201 & \tropzero \\
  -100 & \tropzero & -400 & \tropzero & \tropzero & -201  
\end{pmatrix} \\
G= &  \begin{pmatrix}
  \tropzero & -300 & \tropzero  &   0 &   -1 & \tropzero \\
     0 & \tropzero & -300 & \tropzero & \tropzero & -101 \\
  \tropzero &    0 & \tropzero & -201 & -301 & \tropzero \\
  -300 & \tropzero &    0 & \tropzero & \tropzero & -401 \\
  -200 & \tropzero & -500 & \tropzero & \tropzero & -301 \\
  \tropzero & -100 & \tropzero & -301 & -401 & \tropzero 
\end{pmatrix},\\
CS^{43\modd{4}}R[G]= &  \begin{pmatrix}
  \tropzero & -300 & \tropzero  &   0 &   -1 & \tropzero \\
     0 & \tropzero & -300 & \tropzero & \tropzero & -101 \\
  \tropzero &    0 & \tropzero & -201 & -202 & \tropzero \\
  -300 & \tropzero &    0 & \tropzero & \tropzero & -401 \\
  -200 & \tropzero & -500 & \tropzero & \tropzero & -301 \\
  \tropzero & -100 & \tropzero & -301 & -302 & \tropzero
\end{pmatrix} \\
    H=&  \begin{pmatrix}
     0 & \tropzero & -300 & \tropzero & \tropzero & -101 \\
  \tropzero &    0 & \tropzero & -201 & -301 & \tropzero \\
  -300 & \tropzero &    0 & \tropzero & \tropzero & -401 \\
  \tropzero & -300 & \tropzero &    0 &   -1 & \tropzero \\
  \tropzero & -200 & \tropzero & -401 & -501 & \tropzero \\
  -400 & \tropzero & -100 & \tropzero & \tropzero & -501 
\end{pmatrix},\\
CS^{44\modd{4}}R[H]= & \begin{pmatrix}
     0 & \tropzero & -300 & \tropzero & \tropzero & -101 \\
  \tropzero &    0 & \tropzero & -201 & -202 & \tropzero \\
  -300 & \tropzero &    0 & \tropzero & \tropzero & -401 \\
  \tropzero & -300 & \tropzero &    0 &   -1 & \tropzero \\
  \tropzero & -200 & \tropzero & -401 & -402 & \tropzero \\
  -400 & \tropzero & -100 & \tropzero & \tropzero & -501
\end{pmatrix}
\end{align*}
}\fi

Combining all classes gives us a family of words covering all lengths greater than $9$ such that any product made using these words will not be CSR. Therefore no transient can exist as there is no upper limit to the lengths of these words.

\subsection{Critical graph is not connected} 
\label{cpt:multiplescc}

For this counterexample we now consider a digraph with multiple critical components $\crit_{1},\ldots,\crit_{m}$ which are each strongly connected components with respective cyclicities $\gamma_{1},\ldots,\gamma_{m}$.
\if{
We denote the set of critical nodes for each critical subgraph $\crit_{i}$ as $\critnodes(i)$ with respective set sizes $q_{1},\ldots,q_{m}$. We will require the following definition and notation before we continue.

\begin{definition} \label{d:middle}
Consider a trellis digraph $\trellis(\Gamma(k))$. By a \emph{middle walk} connecting $\critnodes(i)$ to $\critnodes(j)$ on $\trellis(\Gamma(k))$ we mean a walk on $\trellis(\Gamma(k))$ connecting node $x:s$ to $y:t$ where, $x \in \critnodes(i)$, $y\in \critnodes(j)$,
$s$ is the first and last time the walk leaves $C_{i}$, $t$ is he first and last time the walk enters $C_{j}$, where $s$ and $t$ are such that $0<s\leq t<k$. 
The walk does not traverse any other critical subgraph $C_{l}$ otherwise it would become a middle walk connecting $x \in \critnodes(i)$ to $z \in\critnodes(l)$ of length $t_{1}<t$.
\end{definition}

\begin{notation} The following notation is required for the digraphs with multiple critical components:
\begin{itemize} 
    \item [] $\sigma_{\critnodes(i),\critnodes(j)}$ : the weight of the optimal path on $\digr(\Asup)$ connecting a node in $\critnodes(i)$ to a node in $\critnodes(j)$;
    \item [] $u^{\ast}_{\critnodes(i),\critnodes(j)}$ : the maximal weight of walks on $\trellis(\Gamma(k))$ connecting a node in $\critnodes(i)$ to a node in $\critnodes(j)$.
\end{itemize}
\end{notation}
}\fi

\begin{asssumption}{P}{3} \label{a:multiplescc}
$\crit(\mathcal{X})$ is composed of multiple strongly connected components $\crit_{1},\ldots,\crit_{m}$ where the component $\crit_{i}$ has cyclicity $\gamma_{i}$. The cyclicity of $\digr(\mathcal{X})$ is $\lcm_{i}(\gamma_{i})$, which is the same as the cyclicity of $\crit(\mathcal{X})$.
\end{asssumption}

Let us now show a counterexample, which demonstrates that, for the case of several critical components, we cannot have any bounds after which 
the product becomes CSR in terms of $\Asup$ and 
$\Ainf$. The reason is that the non-critical parts of optimal walks whose weights are the entries of $C$ and $R$ cannot be separated in time: in general, they will use the same letters, and such walks on the symmetric extension of $\trellis(\Gamma(k))$ cannot be transformed back to the walks on $\trellis(\Gamma(k))$.

Let $\digr(\mathcal{X})$ be the four node digraph with the following structure:

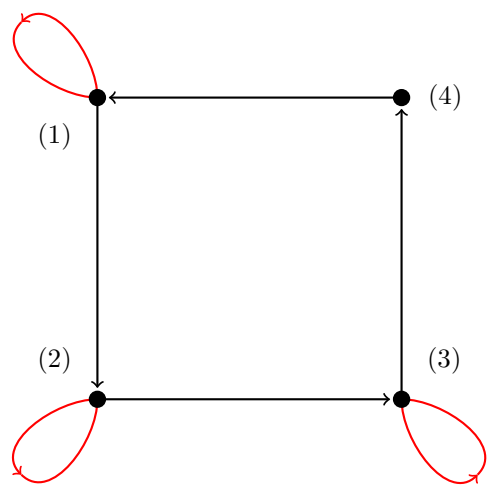
\begin{center}
 \begin{tikzpicture}[thick]
\coordinate (1) at (-2,2);
\coordinate (2) at (-2,-2);
\coordinate (3) at (2,-2);
\coordinate (4) at (2,2);
\coordinate (1a) at (-3,3);
\coordinate (2a) at (-3,-3);
\coordinate (3a) at (3,-3);

\draw [black, ->,shorten >= 0.15cm]   (1) to (2);
\draw [black, ->,shorten >= 0.15cm]   (2) to (3);
\draw [black, ->,shorten >= 0.15cm]   (3) to (4);
\draw [black, ->,shorten >= 0.15cm]   (4) to (1);

\draw [red, ->]   (1) to[out=90,in=45] (1a);
\draw [red, ->]   (2) to[out=180,in=135] (2a);
\draw [red, ->]   (3) to[out=-90,in=-135] (3a);
\draw [red]   (1a) to[out=-135,in=180] (1);
\draw [red]   (2a) to[out=-45,in=-90] (2);
\draw [red]   (3a) to[out=45,in=0] (3);

\draw[fill=black] (1) circle(0.1);
\node[label=below left:{$(1)$}] at (1) {a};
\draw[fill=black] (2) circle(0.1);
\node[label=above left:{$(2)$}] at (2) {a};
\draw[fill=black] (3) circle(0.1);
\node[label=above right:{$(3)$}] at (3) {a};
\draw[fill=black] (4) circle(0.1);
\node[label=right:{$(4)$}] at (4) {a};
\end{tikzpicture}    
\end{center}
along with the following associated weight matrix
\begin{equation*}
    A = \begin{pmatrix}
     0 & A_{12}  &   \tropzero  & \tropzero     \\
  \tropzero  &   0   & A_{23}  &  \tropzero  \\
  \tropzero & \tropzero  &  0  & A_{34}  \\
  A_{41}  & \tropzero & \tropzero &    \tropzero  
    \end{pmatrix}.
\end{equation*}

For this digraph we have a the critical subgraph comprised of three separate loops at nodes $1$,$2$ and $3$. There is also a cycle of length $4$ which means the cyclicity of the digraph is $1$. We are going to present a class of words of infinite length such that the matrix generated by this class of words is not CSR.

We introduce a semigroup of tropical matrices with two generators $\mathcal{X} = \{A_{1},A_{2}\}$ 
where $A_{1}$ to $A_{2}$ are
\begin{equation*}
     A_{1} = \begin{pmatrix}
     0 & -100  &   \tropzero  & \tropzero    \\
  \tropzero  &   0   & -100  &  \tropzero  \\
  \tropzero & \tropzero  &   0 & -100 \\
  -100 & \tropzero & \tropzero &    \tropzero  
    \end{pmatrix},\quad 
    A_{2} = \begin{pmatrix}
     0 & -1  &   \tropzero  & \tropzero     \\
  \tropzero  &   0   & -1  &  \tropzero  \\
  \tropzero & \tropzero  &   0 & -100 \\
  -100 & \tropzero & \tropzero &    \tropzero 
  \end{pmatrix}
 \end{equation*}
and the class of the words that we will consider is 
$(1)^t2$, where $t\geq 2$. In other words we will consider a set of matrices of the form $U=(A_1)^tA_2$ (the actual value of $t\geq 2$ will not matter to us).

We have:  $U_{1,2}=-1$ (as the weight of the walk
$\underbrace{11\ldots 1}_{t+1}2$), 
$U_{2,3}=-1$ (as the weight of the walk $\underbrace{22\ldots 2}_{t+1}3$),and therefore $(CS^{t+1}R[U])_{1,3}= U^2_{1,3}=U_{1,2}\otimes U_{2,3}=-2$,  but $U_{1,3}=-101$ (as the weight of the walk 
$1\underbrace{22\ldots 2}_t3$).

Similarly, we can also look at the entry $U_{4,3}$. 
Then we have $U_{4,2}=-101$ (as the weight of the walk
$4\underbrace{11\ldots 1}_t2$), $U_{2,3}=-1$ and hence 
$(CS^{t+1}R)_{4,3}=(USU)_{4,3}=U_{4,2}\otimes U_{2,3}=-102$, but $U_{4,3}=-201$ (as the weight of the walk 
$41\underbrace{22\ldots 2}_{t-1}3$).

Here is an example of the word from the class for 
$t=10$ and the corresponding $CSR$
\begin{equation*}
W=\begin{pmatrix}
     0  &  -1 & -101 & -300 \\
  -300  &   0 &   -1 & -200 \\
  -200 & -201 &    0 & -100 \\
  -100 & -101 & -201 & -400
  \end{pmatrix}, \qquad CS^{11\modd{1}}R[W]= \begin{pmatrix}
   0  &  -1  &  -2 & -201 \\
  -201  &   0 &   -1 & -101 \\
  -200 & -201 &    0 & -100 \\
  -100 & -101 & -102 & -301
  \end{pmatrix}.    
\end{equation*}

Therefore any matrix product of length greater than $3$ which has been made following this word will not be CSR. Hence there can be no upper bound to guarantee the CSR decomposition in this case.

\section*{Acknowledgments}
The authors are grateful to Oliver Mason, Glenn Merlet, Thomas Nowak and Stephane Gaubert with whom the ideas of this paper were discussed.

\bibliographystyle{siamplain}
\bibliography{references}

\end{document}